\documentclass[10pt]{article}
\usepackage{amssymb,color}
\usepackage{amsmath,graphicx,enumerate,textcomp}
\usepackage[abbrev]{amsrefs}
\usepackage[a4paper,top=1.in, bottom=1.5in, left=.8in, right=.8in, headsep=0.in, centering]{geometry}
\usepackage{amssymb,amsfonts,amsthm}
\numberwithin{equation}{section}
\usepackage{lipsum}
\newcommand{\eps}{\varepsilon}

\newtheorem{theorem}{Theorem}[section]
\newtheorem{lemma}{Lemma}[section]
\newtheorem{remark}{Remark}[section]
\newtheorem{proposition}{Proposition}[section]

\newtheorem{corollary}{Corollary}[section]

\title{Gelfand-type problem for turbulent jets.
}

\author {Peter  V. Gordon
\thanks{Department of Mathematical Sciences, 
Kent State University,
 Kent, OH 44242, USA. E-mail: {\tt gordon@math.kent.edu}}
\and Vitaly Moroz
\thanks{
Department of Mathematics, 
Computational Foundry,
Swansea University,
Fabian Way, Swansea SA1 8EN,
Wales, UK. E-mail: {\tt v.moroz@swansea.ac.uk}}
\and Fedor Nazarov
\thanks{Department of Mathematical Sciences, 
Kent State University,
Kent, OH 44242, USA. E-mail: {\tt nazarov@math.kent.edu}}
}

\begin{document}
\maketitle

\begin{abstract}
We consider the model of auto-ignition (thermal explosion) of a free round reactive turbulent jet introduced in \cite{GHH18}. This model falls into the general class of Gelfand-type problems and constitutes a
boundary value problem for a certain semi-linear elliptic equation that depends on
two parameters: $\alpha$ characterizing the flow rate and $\lambda$  (Frank-Kamentskii parameter) characterizing the strength of the reaction.  Similarly to the classical Gelfand problem, this equation admits a solution when the Frank-Kametskii parameter $\lambda$ does not
exceed some critical value $\lambda^*(\alpha)$ and admits no solutions for larger values of $\lambda$. We obtain the sharp asymptotic behavior of the critical Frank-Kamenetskii parameter in the 
strong flow limit ($\alpha\gg1$). We also provide a detailed description of the extremal solution (i.e., the solution corresponding to $\lambda^*$) in this regime.
\end{abstract}



\section{Introduction}

In this paper we are concerned with the existence and quantitative properties of solutions of  the following problem:
\begin{eqnarray}\label{eq:1}
\left\{
\begin{array}{lll}
-\Delta u-\alpha r \varphi(r) \frac{\partial }{\partial r} u=\lambda \psi(r) f(u) & \mbox{in} & B, \\
 u>0 &\mbox{in} & B, \\
 u=0 & \mbox{on} & \partial B,
\end{array}
\right.
\end{eqnarray}
where $B$ is the unit disk in $\mathbb{R}^2$ centered at the origin,  $\lambda>0, ~\alpha > 0$ are parameters,
$f: [0,\infty) \to (0,\infty)$ is a $C^1$ convex non-decreasing  function satisfying:
\begin{eqnarray} \label{eq:2}
\int_0^{\infty}\frac{ds}{f(s)}<\infty,
\end{eqnarray}
$\varphi(r),\psi(r)$ are  non-negative, non-increasing Lipshitz continuous functions on $[0,1]$
satisfying $\varphi(0)=\psi(0)=1$.  Moreover, we assume that $\varphi(r)$ is positive in $[0,1).$ 
In addition we  assume that
\begin{eqnarray}\label{eq:3}
\int_0^1 M(s) ds <\infty,
\end{eqnarray}
where 
\begin{eqnarray}\label{eq:4}
M(s):=\max_{r\in [0,s]} \frac{\psi(r)}{\varphi(r)}.
\end{eqnarray}

Equation \eqref{eq:1} was recently introduced in \cite{GHH18} as a model of the autoignition event of  free round reactive turbulent jets.
In the context of this model, $u$,  $f$,  $\varphi$, and $\psi$ 
are respectively the (appropriately normalized) temperature, reaction rate, flow velocity profile, and product of concentrations
of the oxidizing and reactive components over cross-sections of the jet, while
$\lambda$ is the Frank-Kamentskii parameter representing the ratio of the heat release of
the reaction and the thermal conductivity  and $\alpha$ is the ratio of the injection velocity
and the thermal conductivity.

The typical examples considered in the physical literature (e.g. \cite{Abram,Shl}) are
$f(u)=\exp(u)$, $\varphi(r)=\exp(-4r^2)$ or $\varphi(r)=(1-r^{3/2})^2$, and
$\psi=\varphi^{2{\rm Sc}},$ where ${\rm Sc}$ is a Sclihting number which for round turbulent jets is 
${\rm Sc}\approx 0.75$, in which cases our assumptions \eqref{eq:2} and \eqref{eq:3}  are easily seen to be satisfied.

Problem \eqref{eq:1} falls into the general class of Gelfand-type problems. 
The classical Glefand problem can be obtained from  problem \eqref{eq:1}
by removing the advection term, making the right hand side of the equation independent of
$r,$ that is,  by setting $\alpha=0$ and $\psi=1,$ and replacing the unit disk $B$ in $\mathbb{R}^2$  by a general bounded
domain
in $\mathbb{R}^n.$  That problem was introduced in  1938 by Frank-Kamenetskii as a model of thermal explosion in a combustion  vessel with ideally thermally conducting walls
(see \cite{FK,Sem,ZBLM} for more details), but became known
in the mathematical community due to the chapter written by Barenblatt in a famous review of Gelfand
\cite{Gelfand}. The general properties of solutions of the classical Gelfand problem were studied quite extensively in both mathematical and physical literature, see book \cite{stable} for a review of results.
The problem considered in this paper inherits many nice features of the classical Gelfand
problem.  The following proposition summarizes the properties of solutions of problem \eqref{eq:1} relevant to the present work.

\begin{proposition} \label{p:1}

For fixed $\alpha>0,$ there exists an extremal value
of the Frank-Kamenetskii parameter $\lambda^*=\lambda^*(\alpha)\in (0,\infty)$ such that:

\medskip

\noindent i) Problem \eqref{eq:1} admits a unique minimal (i.e. smallest) classical positive solution 
$u_{\lambda,\alpha}$ for $\lambda\in(0,\lambda^*);$  

\medskip

\noindent ii) Problem \eqref{eq:1} admits a unique extremal solution $u^*_{\alpha}$ defined as
\begin{eqnarray}\label{eq:weaks}
u^*_{\alpha}(x):=\lim_{\lambda\to\lambda^*} u_{\lambda,\alpha}(x),
\end{eqnarray}
which is also classical.

\medskip

\noindent iii) Minimal solutions of \eqref{eq:1} for $\lambda\in(0,\lambda^*]$ are radially symmetric, strictly decreasing and satisfy the semi-stability condition
\begin{eqnarray}\label{eq:semistab}
\int_{B} |\nabla \eta|^2 \mu dx\ge\lambda \int_{B} \psi f'(u_{\lambda,\alpha} ) \eta^2 \mu dx,\quad \forall \eta \in
H^1_0(B),
\end{eqnarray}
 where 
 \begin{eqnarray}\label{eq:app1a}
\mu(r)=\exp\left(\alpha\int_0^r s\varphi(s)ds\right);
\end{eqnarray}

\medskip\noindent 
iv) There are no solutions for \eqref{eq:1} when  $\lambda>\lambda^*$.
\end{proposition}
The proposition above is quite standard. We will present a sketch of the proof in the next section for completeness.

 In the context of the autoignition problem, the extremal value $\lambda^*(\alpha)$ and  the extremal solution $u^*_{\alpha}$ play a very special  role. Indeed, in the context of the theory developed
 in \cite{GHH18}, as any theory based on Frank-Kamenetskii approach, the existence of a solution for \eqref{eq:1} indicates autoignition failure. From physical
 standpoint, that means that the reactive component undergoes partial oxidation, which
 results in establishing a self-similar temperature profile given by the minimal solution
 of \eqref{eq:1}. In contrast, the absence of a solution for \eqref{eq:1} indicates successful autoignition.
 Therefore, $\lambda^*(\alpha)$ determines the boundary between the successful autoignition and
 the absence thereof.  The extremal value of the Frank Kamentskii parameter $\lambda^*(\alpha)$
 indicates the maximal reaction intensity for a given flow rate for which auto-ignition does not take
 place. 
  The extremal solution determines the maximal possible self-similar
 profile.
 
 In practical applications $\alpha\gg 1$ and hence one needs to understand the behavior
 of $\lambda^*(\alpha)$ for large $\alpha$.
 This observation raises the question of asymptotic behavior of $\lambda^*(\alpha)$ as $\alpha\to\infty$.
 Lower and upper bounds on $\lambda^*(\alpha)$ as $\alpha\to\infty$ were derived in \cite[Theorem 3.1]{GHH18} in the special case $f(u)=\exp(u).$
 In this paper we establish a sharp asymptotic of $\lambda^*(\alpha)$ and give quite precise 
 description of an extremal solution in this limit for a general class of nonlinearities $f(u)$
 and functions $\varphi, \psi$  under  very mild  regularity assumptions. Our main results are given by the following 
 theorems. 
 
 The first theorem gives sharp asymptotic for $\lambda^*(\alpha)$ for large values of $\alpha$:

\begin{theorem}  \label{t:1}
Let
\begin{eqnarray} \label{eq:9}
 K:=\int_{0}^{\infty} \frac{ds}{f(s)}.
\end{eqnarray}
Then,
\begin{eqnarray}\label{eq:a1}
\lim_{\alpha\to\infty} \lambda^*(\alpha) \left( \frac{2K \alpha} {\log \alpha} \right)^{-1}=1.
\end{eqnarray}
\end{theorem}

The second theorem provides details of the behavior of the extremal solution when $\alpha\gg1:$

\begin{theorem} \label{t:2} 
Let $u_{\alpha}^*$ be the extremal solution of \eqref{eq:1}.
Then, as $\alpha \to \infty,$ we have
\begin{eqnarray}
&& i) ~u^*_{\alpha}(x) \to 0 \quad \forall x\in B\setminus\{0\}, \qquad  u^*_{\alpha} (0) \to \infty,
\end{eqnarray}
and
\begin{eqnarray}\label{eq:t222}
&& ii) ~\int_{B} u^*_{\alpha}(x) dx \to 0, \qquad \int_{B} \psi(x) f(u_{\alpha}^*(x))dx\to f(0) \int_{B} \psi(x)dx.
\end{eqnarray}
\end{theorem}

While Proposition \ref{p:1} ensures that extremal solutions of problem \eqref{eq:1} are bounded,
the establishment of a reasonable uniform upper bound appeared to be a difficult task.
However, in case of sufficiently regular non-linearities  or non-linearities with sufficiently fast growth at infinity we have
the following result.

\begin{theorem} \label{t:3} Assume that there exist  constants $0<c_0<1,$  $c_1>1$ and $t_0>0$
such that 
\begin{eqnarray}\label{eqt:1}
f(t_2)\ge c_1 f(t_1), \quad t_2 >t_1>  t_0,
\end{eqnarray}
implies
\begin{eqnarray}\label{eqt:2}
(1-c_0) f^{\prime}(t_2)\ge f^{\prime}(t_1).
\end{eqnarray}
Then,
\begin{eqnarray}\label{eqt:3}
u^*_{\alpha}(0) \le c A,
\end{eqnarray}
where $A$ is the solution of
\begin{eqnarray}\label{eqt:4}
f^{\prime} (A)= c \log \alpha
\end{eqnarray}
and $c>0$ is some constant independent of $\alpha.$

Moreover,
\begin{eqnarray}\label{eqt:5}
\int_B (u_{\alpha}^*(x))^pdx \to 0 \quad \mbox{as} \quad \alpha \to \infty,
\end{eqnarray}
for any $1\le p<\infty.$

\end{theorem}

\begin{remark}
It is easy to check that the assumptions of Theorem \ref{t:3} are satisfied for most typical nonlinearities
such as $f(u)=\exp(u)$ and $f(u)=(1+u)^p$ for $p>1.$ Moreover, they are also satisfied provided
$f$ is a $C^2$ function   such that  $f^{\prime}(s), f^{\prime\prime}(s)>0$ 
and $f^{\prime\prime}(s)$ is strictly increasing on $(0,\infty)$
(see Lemma \ref{l:proizv}).
\end{remark}

\begin{remark} In the case of exponential nonlinearity, we also have that
\begin{eqnarray}\label{eq:lexp}
\lambda^*(\alpha)=\frac{2\alpha}{\log \alpha}\left(1+O\left(\frac{1}{\sqrt{\log \alpha}}\right)\right)
\end{eqnarray}
and
\begin{eqnarray}
u_{\alpha}^*(0)=O(\log(\log \alpha)),
\end{eqnarray}
see Lemma \ref{l:exp}.
These results are consistent with the formal asymptotic extremal solution of \eqref{eq:1} obtained in \cite{GHH18}.
\end{remark}

The paper is organized as follows. In section 2 we are setting up the stage by giving necessary definitions, providing standard results and introducing rescaling which makes the analysis more convenient. In sections 3,4,5  we give  proofs of theorems 1.1, 1.2, and 1.3 respectively. 

\section{Preliminaries}

In this section we  outline a proof of Proposition \ref{p:1} and  introduce alternative forms of problem \eqref{eq:1},  which will be used in the later sections.

First we observe that problem \eqref{eq:1} can be written in the divergence form.
Indeed one can verify by direct computations that  \eqref{eq:1} can be rewritten as 
follows
\begin{eqnarray}\label{eq:app1}
\left\{
\begin{array}{ccc}
-\nabla\cdot(\mu(r) \nabla u)=\lambda \mu(r)\psi(r) f(u) & \mbox{in} & B,\\
u>0 & \mbox{in} &B,\\
u=0  & \mbox{on} & \partial B,
\end{array}
\right.
\end{eqnarray}
where  $\mu$ is given by \eqref{eq:app1a}. We note that $\mu\in C^{1,\omega}$  ($0<\omega<1$) as
follows from  definition and properties of $\varphi$, while $\psi$ is Lipshitz continuous.
The results presented in this
section deal with the situation when $\alpha>0$ is fixed. Therefore, we will omit subscript
$\alpha$ when referring to minimal and extremal solutions of \eqref{eq:app1}. 

Proposition \ref{p:1} is basically a compilation of well known results (or their minimal
adaptations) presented in \cite{Brezis96,Satt71,Cabre06,CR75,KK74,keller67,nedev,sasha}.
It  follows from the sequence of Lemmas \ref{l:int1}-\ref{l:int5} presented below.

A proof of Proposition \ref{p:1} is based on the construction of sub and super-solutions for problem
\eqref{eq:app1}.
Following \cite{PW,Evans}, we define
a classical positive super-solution of  \eqref{eq:app1} as  a function $\bar u\in C^2(B)\cap C(\bar B)$ 
 positive in $B$ such that
\begin{eqnarray}\label{eq:app2}
\left\{
\begin{array}{ccc}
-\nabla\cdot(\mu(r) \nabla \bar u)\ge \lambda \mu(r)\psi(r) f(\bar u) & \mbox{in} & B,\\
\bar u\ge 0  & \mbox{on} & \partial B,
\end{array}
\right.
\end{eqnarray}
and classical non-negative sub-solution of \eqref{eq:app1} as a function $\underline u\in C^2(B)\cap C(\bar B)$ non-negative in $\bar B$
such that \begin{eqnarray}\label{eq:app2sub}
\left\{
\begin{array}{ccc}
-\nabla\cdot(\mu(r) \nabla \underline u)\le \lambda \mu(r)\psi(r) f(\underline  u) & \mbox{in} & B,\\
\underline u= 0  & \mbox{on} & \partial B,
\end{array}
\right.
\end{eqnarray}
We note that under the assumptions  of this paper $\underline u=0$ is always a sub-solution.

\begin{lemma} \label{l:int1}
Assume \eqref{eq:app1} admits positive,
 classical super-solution $\bar u$. Then \eqref{eq:app1} admits a unique minimal,  
positive classical solution $u_{\lambda}\in C^{2,\omega} (\bar B).$
This minimal solution is radially symmetric, strictly decreasing
and bounded by $\bar u$ from above.
\end{lemma}

\begin{proof}
The minimal solution $u_{\lambda}$ of \eqref{eq:app1} is obtained by a construction using monotone iteration arguments.
Namely, we consider a sequence of functions $\{u_n\}_{n=0}^{\infty}$ with 
$u_0=0$ and $u_n$ defined as
\begin{eqnarray}\label{eq:app2c}
\left\{
\begin{array}{ccc}
-\nabla\cdot(\mu(r) \nabla  u_{n})+\Omega u_n = \lambda \mu(r)\psi(r) f( u_{n-1}) +\Omega u_{n-1}& \mbox{in} & B,\\
u_{n}= 0  & \mbox{on} & \partial B,
\end{array}
\right.
\end{eqnarray}
for $n\ge 1$, where $\Omega>0$ is an arbitrary constant.

As follows from \cite[Theorem 6.14]{GT} for each $n$ problem \eqref{eq:app2c} 
admits a unique solution $u_n\in C^{2,\omega}(\bar B)$. Each $u_n$
is radial, as follows from the uniqueness.
We now define the minimal solution of \eqref{eq:app1} as
\begin{eqnarray}\label{eq:app2d}
u_{\lambda}(x):=\lim_{n\to\infty} u_n(x).
\end{eqnarray}
By \cite[Theorem 2.1]{Satt71} we have that  $u_{\lambda}$ defined by \eqref{eq:app2d} satisfies  $\bar u\ge u_{\lambda}>0$
in $B,$  
belongs to $C^{2,\omega}(\bar B)$ and solves \eqref{eq:app1} classically.
Moreover, since each $u_n$ is radially symmetric we have that $u_{\lambda}$ is also
radially symmetric so that  $u_{\lambda}(x)=u_{\lambda}(|x|)=u_{\lambda}(r)$.
 Consequently, any minimal solution constructed above satisfies
\begin{eqnarray}\label{eq:app1b}
\left\{
\begin{array}{ll}
-\frac{d}{dr} \left( r \mu(r) \frac{d}{dr}u_{\lambda}\right)=\lambda r \mu(r) \psi(r) f(u_{\lambda})
& 0<r<1, \\
 \frac{d}{dr}u_{\lambda}(0)=0,& u_{\lambda} (1)=0.
\end{array}
\right.
\end{eqnarray}
 Integrating \eqref{eq:app1b} we also have that for $r\in(0,1],$
\begin{eqnarray}
\frac{d}{d r} u_{\lambda}(r) = -\frac{\lambda}{r\mu(r)}\int_0^r \mu(s)\psi(s) f(u_{\lambda}(s))sds  <0,
\end{eqnarray}
and hence $u_{\lambda}$  is strictly decreasing. 
\end{proof}

The following lemma uses the notion of a weak solution. Similarly  to \cite{Brezis96}, we define a weak solution of \eqref{eq:app1} as  a non-negative function $u\in L^1(B)$ such
that $\psi f(u) { \rm dist}_{\partial B} \in L^1(B),$ where ${\rm dist}_{\partial B}(x)$ is the distance from $x$ to the boundary of $B$ and
\begin{eqnarray}
-\int_{B} u \nabla\cdot(\mu \nabla \zeta)dx=\int_{B} \psi \mu f(u) \zeta dx,
\end{eqnarray}
for all $\zeta\in C^2(\bar B)$ with $\zeta =0$ on $\partial B.$

\begin{lemma} \label{l:int2}
Problem \eqref{eq:app1} admits a minimal classical solution $u_{\lambda}$
for $0<\lambda<\lambda^*<\infty$. Moreover, the extremal solution $u^*$ defined by \eqref{eq:weaks} is a weak solution of
\eqref{eq:app1}.

\end{lemma}
\begin{proof}

First observe that $u_{\lambda}$ is a non-decreasing  function of $\lambda.$ 
This follows from the fact that
$u_{\lambda^{\prime}}$ is a super-solution for  problem \eqref{eq:app1} with
 $\lambda<\lambda^{\prime}.$ Hence, if \eqref{eq:app1} with $\lambda=\lambda^{\prime}$
 admits a classical  solution, then \eqref{eq:app1} admits a classical solution for $\lambda\in(0,\lambda^{\prime}]$.

Next, let $\tau$ be a solution of 
\begin{eqnarray}\label{eq:app4}
\left\{
\begin{array}{ccc}
-\nabla\cdot(\mu \nabla  \tau )= 1 & \mbox{in} & B,\\
\tau = 0  & \mbox{on} & \partial B.
\end{array}
\right.
\end{eqnarray}
It is easy to see that $\tau$ is a super-solution for \eqref{eq:app1} provided
$\lambda \le \big(\mu(1) f(\tau(0))\big)^{-1}$. This establishes the existence of a minimal solution
for small enough $\lambda$.

Now let us show that $\lambda^*<\infty,$ which is done by a slight adaptation
of \cite[Lemma 5]{Brezis96}. This adaptation is needed because $\psi$ might be
zero in some portion of $B$. 
By convexity of $f$ we have that there is $\eps>0$ such that $f(s)>\eps s$ for $s\ge 0.$
Hence
\begin{eqnarray}\label{eq:app8}
-\nabla \cdot (\mu \nabla u) \ge \eps \lambda  \psi \mu u \quad \mbox{in} \quad B.
\end{eqnarray}
Let $\kappa_1, \xi_1$ be the principal eigenvalue and the corresponding eigenfunction
of  the generalized eigenvalue problem
\begin{eqnarray}
\left\{
\begin{array}{ccc}
-\nabla \cdot (\mu \nabla \xi)=\kappa \psi \mu \xi  & \mbox{in} & B, \\
\xi=0 & \mbox{on} & \partial B.
\end{array}
\right.
\end{eqnarray}
Variational  characterization of $\kappa_1$ and arguments identical to these of
\cite[Theorem 8.38]{GT} show that $\kappa_1>0$ and $\xi_1>0$ in $B$. 

Multiplying \eqref{eq:app8} by $\xi_1$ and integrating by parts we obtain
that 
\begin{eqnarray}
\kappa_1 \int_B \psi \mu u \xi_1 dx \ge \eps \lambda \int_B \psi \mu u \xi_1 dx.
\end{eqnarray}
Thus,
$\eps \lambda \le \kappa_1$ and hence $\lambda^*\le \kappa_1/\eps.$ 

Finally, proceeding as in the proof of \cite[Lemma 5]{Brezis96} with the only modification
that $-\Delta$ is replaced by $-\nabla\cdot( \mu \nabla (\cdot)),$ we recover that $u^*$ is
a weak solution of \eqref{eq:app1}.
\end{proof}

\begin{lemma} \label{l:int3}
Problem \eqref{eq:app1} admits neither classical, nor weak solutions for $\lambda>\lambda^*$.
\end{lemma}
Proof of this lemma is a line by line adaptation of  \cite[Theorem 3]{Brezis96} 
with the only difference that $-\Delta$ is replaced by $-\nabla\cdot( \mu \nabla (\cdot))$ 
and is omitted here.

As a next step we give a proof of the semi-stability condition \eqref{eq:semistab}. The proof is  similar to the
one of \cite[Theorem 4.2]{Satt71}, with a slight modification, which is required due to
the restriction on regularity of the nonlinear term in \eqref{eq:app1}.
 
\begin{lemma} \label{l:int4}
The semi-stability condition \eqref{eq:semistab} holds for any minimal
solution $u_{\lambda}$ with $\lambda\in(0,\lambda^*]$.
\end{lemma}
\begin{proof}
Let 
\begin{eqnarray}
{\cal L} \eta= -\nabla\cdot(\mu \nabla \eta)-\lambda \mu\psi f^{\prime} (u_{\lambda}) \eta, \quad
{\cal \tilde  L} \tilde \eta= -\nabla\cdot(\mu \nabla \tilde \eta)-\lambda \mu\psi  f_{\eps}^{\prime} (u_{\lambda}) \tilde \eta,
\end{eqnarray}
where $ f_{\eps}^{\prime}$ is $C^{\omega}$ function such that
\begin{eqnarray}
\left | f^{\prime}(s)-f_{\eps}^{\prime}(s)\right|<\eps \quad \mbox{for} \quad s\in[0,u_{\lambda}(0)],
\end{eqnarray}
for some $\eps>0$
and set $\lambda_1,\tilde \lambda_1$  to be the principal eigenvalues 
of ${\cal L}$ and ${\cal \tilde L}$ respectively.

Assume first that $\lambda<\lambda^*$. We  claim that $ \lambda_1\ge 0.$
To show that, we note that the  first eigenfunction of ${\cal \tilde  L}$, 
 $\tilde \eta_1$
is positive in $B$ and $\tilde \eta_1\in C^{2,\omega}(\bar B)$
as follows from \cite[Theorem 6.15, Theorem 8.38]{GT}. Choosing the normalization
in such a way that  $||\tilde \eta_1||_{C^1(\bar B)}=1,$ we observe that
 $\tilde u=u_{\lambda}-\eps \tilde \eta_1$ is in  
$C^{2,\omega}(\bar B)$ and positive in $B,$ provided that $\eps$ is sufficiently small.
The latter is guaranteed by the fact that the normal derivative of $u_{\lambda}$ on the boundary is strictly positive as follows from  Hopf's lemma \cite[Chapter 2, Theorem7]{PW}.

We next observe that
\begin{eqnarray}
-\nabla\cdot( \mu \nabla \tilde u)-\lambda \mu \psi f(\tilde u)=
-\eps \tilde \lambda_1 \tilde \eta_1 + \lambda \mu\psi R_{\eps},
\end{eqnarray}
with
\begin{eqnarray}
R_{\eps}= \left[ f(u_{\lambda})-\eps f^{\prime}(u_{\lambda})\tilde \eta_1 -f(u_{\lambda}-\eps\tilde\eta_1)\right]
 +\eps \left[ f^{\prime}(u_{\lambda}) -f^{\prime}_{\eps} (u_{\lambda}) \right] \tilde \eta_1.
\end{eqnarray}
Clearly $|R_{\eps}|=o(\eps \eta_1)$ as $\eps\to 0$.  Thus, if $\tilde \lambda_1<0,$ then $\tilde u$ is a classical 
positive  super-solution
of \eqref{eq:app1} strictly below $u_{\lambda}$ in $B,$ which contradicts the minimality of $u_{\lambda}$. 
Hence, $\tilde \lambda_1\ge 0$.

Next observe that $\lambda_1,\tilde \lambda_1$ admit the variational characterization
\begin{eqnarray}
\lambda_1=\inf_{\eta\in {\cal A}} \int_B( |\nabla \eta|^2-\lambda \psi f^{\prime}(u_{\lambda}) \eta^2)\mu
dx,  \quad \tilde \lambda_1=\inf_{\eta\in {\cal A}} \int_B( |\nabla \eta|^2-\lambda \psi f_{\eps}^{\prime}(u_{\lambda}) \eta^2)\mu dx,
\end{eqnarray}
where ${\cal A}:=\{ \eta \in H_0^1(B) : ||\eta||_{L^2(B)}=1\}.$
This immediately implies that $\lambda_1\ge \tilde \lambda_1 -\lambda\mu(1) \eps.$
Since $\eps$ is arbitrarily  small, we have that $\lambda_1\ge 0$ and hence
\eqref{eq:semistab} holds for $\lambda<\lambda^*$. 
Moreover, since $u^*$ is an increasing point-wise limit of $u_{\lambda},$ it also holds for $\lambda=\lambda^*$.

\end{proof}

\begin{lemma} \label{l:int5}
An extremal solution of \eqref{eq:app1} is classical.
\end{lemma}
\begin{proof}
Taking $\tilde f(s)=f(s)-f(0)$ and setting $\eta=\tilde f(u_{\lambda})$ in \eqref{eq:semistab}
we have that 
\begin{eqnarray}\label{eq:appA}
\int_{B} |\nabla u_{\lambda}| \left(f^{\prime}(u_{\lambda})\right)^2 \mu dx\ge 
\lambda \int_{B}f^{\prime}(u_{\lambda}) \left(\tilde f (u_{\lambda})\right)^2 \mu \psi dx.
\end{eqnarray}
Next taking $g(t)=\int_0^t (f^{\prime}(s))^2 ds$, multiplying first equation in \eqref{eq:app1} by 
$g(u_{\lambda})$ and integrating the result by parts we have
\begin{eqnarray}\label{eq:appB}
\int_{B} |\nabla u_{\lambda}|^2 \left(f^{\prime}(u_{\lambda})\right)^2 \mu dx=
\lambda \int_{B} f(u_{\lambda}) g(u_{\lambda}) \psi \mu dx.
\end{eqnarray}
Combining \eqref{eq:appA}, \eqref{eq:appB}, we have
\begin{eqnarray}\label{eq:nedev}
\int_{B} f^{\prime} (u_{\lambda}) \left( \tilde f(u_{\lambda})\right)^2 \psi \mu dx\le
\int_{B} f(u_{\lambda}) g(u_{\lambda}) \psi \mu dx.
\end{eqnarray}
Using this inequality instead of Eq.(4) in \cite{nedev} and arguing exactly as in
\cite[Theorem 1]{nedev} we conclude that the extremal solution is classical.

\end{proof}

In this paper we  are mostly concerned with minimal solutions of problem \eqref{eq:1}, which
by Proposition \ref{p:1} are classical and radially symmetric. Therefore, we will only consider 
solutions of \eqref{eq:1}
that are radially symmetric.

To study radially symmetric  solutions, it is convenient to
introduce the following rescaling 
\begin{eqnarray}\label{eq:5}
u(r):=v(\sqrt{\alpha} r)=v(y).
\end{eqnarray}
Substituting \eqref{eq:5} into \eqref{eq:1} we get
\begin{eqnarray}\label{eq:6}
\left\{
\begin{array}{ll}
-v^{\prime\prime}-\left( \frac{1}{y}+y \varphi_{\alpha}(y)\right) v^{\prime}=
\frac{\lambda }{\alpha}\psi_{\alpha} \left(y\right) f(v), & 0<y<\sqrt{\alpha}, \\
 v^{\prime}(0)=0,& v(\sqrt{\alpha})=0.
\end{array}
\right.
\end{eqnarray}
Here and below, $(\cdot)^{\prime}=\frac{d}{d y}(\cdot)$ and
\begin{eqnarray}\label{eq:7}
\varphi_{\alpha}(y):=\varphi\left(\frac{y}{\sqrt{\alpha}}\right), \quad \psi_{\alpha}(y):=\psi\left(\frac{y}{\sqrt{\alpha}}\right).
\end{eqnarray}

We now set
\begin{eqnarray}\label{eq:8}
G(v):=\int_{v}^{\infty} \frac{ds}{f(s)}
\end{eqnarray}
and note that $G:[0,+\infty)\to (0,K]$ is a $C^2$ strictly monotone decreasing bijection and hence \eqref{eq:8} implicitly defines the inverse $G^{-1}:(0,K]\to [0,+\infty)$, which is also strictly monotone decreasing.
Differentiating  \eqref{eq:8}, we obtain
\begin{eqnarray} \label{eq:10}
 G^{\prime}=-\frac{v^{\prime}}{f(v)}, \quad
 G^{\prime\prime}=-\frac{v^{\prime\prime}}{f(v)}+\left(\frac{v^{\prime}}{f(v)}\right)^2 f_v(v)=
-\frac{v^{\prime\prime}}{f(v)}+\left( G^{\prime}\right)^2 f_v(v),
\end{eqnarray}
where
\begin{eqnarray}\label{eq:11}
f_v(v)=\frac{d f(v)}{d v}.
\end{eqnarray}
Combining  \eqref{eq:6} and \eqref{eq:10} yields
\begin{eqnarray}\label{eq:12}
\left\{
\begin{array}{ll}
G^{\prime\prime}+\left( \frac{1}{y}+y \varphi_{\alpha}\left(y\right)\right) G^{\prime}=
\frac{\lambda}{\alpha}\psi_{\alpha} \left(y\right) +\left( G^{\prime}\right)^2 f_v(v), & 0<y<\sqrt{\alpha}, \\
  G^{\prime}(0)=0, & G(\sqrt{\alpha})=K.
 \end{array}
\right.
\end{eqnarray}
where $v=G^{-1}(y)$ is implicitly defined by \eqref{eq:8}.

In what follows we will work with both \eqref{eq:6} and \eqref{eq:12} as  alternative versions
of \eqref{eq:1}. 

We next define super-solution for problem \eqref{eq:6} and sub-solution for \eqref{eq:12}.

A function
 $\bar v>0$  on $[0,\sqrt{\alpha})$ belonging to $C^2((0,\sqrt{\alpha}))\cap C([0,\sqrt{\alpha}])$
 is a  positive super-solution of \eqref{eq:6} provided it verifies
\begin{eqnarray}\label{eq:p1}
\left\{
\begin{array}{ll}
-\bar v^{\prime\prime}-\left( \frac{1}{y}+y \varphi_{\alpha}(y)\right) \bar v^{\prime}\ge
\frac{\lambda }{\alpha}\psi_{\alpha} \left(y\right) f(\bar v), & 0<y<\sqrt{\alpha}, \\
 \bar v^{\prime}(0)=0,& \bar v(\sqrt{\alpha})\ge 0.
\end{array}
\right.
\end{eqnarray}

Similarly, a function $\underline G>0 $  on $[0,\sqrt{\alpha}]$ belonging to $C^2((0,\sqrt{\alpha}))\cap C([0,\sqrt{\alpha}])$  is a positive classical sub-solution of \eqref{eq:12} if
 \begin{eqnarray}\label{eq:p2}
\left\{
\begin{array}{ll}
\underline G^{\prime\prime}+\left( \frac{1}{y}+y \varphi_{\alpha}\left(y\right)\right) \underline G^{\prime}\ge
\frac{\lambda}{\alpha}\psi_{\alpha} \left(y\right) +\left( \underline G^{\prime}\right)^2 f_v(v), & 0<y<\sqrt{\alpha} \\
   \underline G^{\prime}(0)=0, &\underline  G(\sqrt{\alpha})\le  K.
 \end{array}
\right.
\end{eqnarray}

The following results follow from Lemma \ref{l:int1}.
\begin{corollary} \label{l:1}
Assume \eqref{eq:6} has a  positive super-solution $\bar v$, then \eqref{eq:6} has a minimal  positive 
classical solution $v_{\lambda}$.
Moreover, $v_{\lambda}(y)\le \bar v(y)$ in $y\in[0,\sqrt{\alpha}]$
\end{corollary}

\begin{corollary}\label{l:2}
Assume that $\underline G $ is a positive  sub-solution of \eqref{eq:12}. Then, the 
function $\bar v$ implicitly defined by \eqref{eq:8} is a positive  super-solution
of \eqref{eq:6}.
\end{corollary}

\section{Asymptotic behavior  of $\lambda^*$.}

In this section we establish an asymptotic behavior of $\lambda^*(\alpha)$ for
sufficiently large $\alpha$ and give a proof of Theorem \ref{t:1}. 
 In what follows  we will work with \eqref{eq:12}, which is an alternative form of \eqref{eq:1}.

We start with the upper bound for $\lambda^*,$ which is given by the following lemma.

\begin{lemma}\label{l:3}
Assume \eqref{eq:1} has a solution. Then, for $\alpha$ sufficiently large, $\lambda^*(\alpha)$ obeys the following upper bound:
\begin{eqnarray}\label{eq:a14a}
\lambda^*(\alpha) \le \frac{2 K \alpha}{\log \alpha}\left(1+\frac{c}{\log{\alpha}}\right),
\end{eqnarray}
where $c>0$ is some constant independent of $\alpha.$
\end{lemma}
\begin{proof}
First we observe that the first equation in \eqref{eq:12} can be rewritten in the divergence form
\begin{eqnarray}\label{eq:a1a}
\left( y\exp\left (\int_0^y s\varphi_{\alpha}(s)ds\right) G^{\prime}(y)\right)^{\prime}=
\left[\frac{\lambda}{\alpha} \psi_{\alpha}(y) + \left( G^{\prime}(y)\right)^2 f_v(v(y)) \right]y\exp\left(\int_0^y s\varphi_{\alpha}(s)ds\right).
\end{eqnarray}
Therefore,
\begin{eqnarray}\label{eq:a2}
\left( y\exp\left (\int_0^y s\varphi_{\alpha}(s)ds\right) G^{\prime}(y)\right)^{\prime}\ge
\frac{\lambda}{\alpha} y\psi_{\alpha}(y) \exp\left(\int_0^y s\varphi_{\alpha}(s)ds\right).
\end{eqnarray}
Integrating \eqref{eq:a2}  from $0$ to $y$  and using the boundary condition at zero in \eqref{eq:12}, we get
\begin{eqnarray}\label{eq:a3}
G^{\prime}(y)\ge \frac{\lambda}{\alpha} \frac{1}{y}\int_0^y z \psi_{\alpha}(z)
\exp \left( -\int_z^y s\varphi_{\alpha}(s) ds \right) dz.
\end{eqnarray}
Next observe that due to the monotonicity of $\psi$ and $\varphi$ we have
\begin{eqnarray}\label{eq:a4}
G^{\prime}(y)\ge  \frac{\lambda}{\alpha} \frac{\psi_{\alpha}(y)}{y} \int_0^y
z\exp\left(-\int_z^y sds\right) dz=
\frac{\lambda}{\alpha} \frac{\psi_{\alpha}(y)}{y}\left(1-\exp\left( -\frac{y^2}{2}\right)\right).
\end{eqnarray}
Since $\psi(0)=1$ and $\psi$ is Lipshitz continuous, we have that
\begin{eqnarray}\label{eq:a5}
\psi(x)\ge 1-\frac{x}{h} ,
\end{eqnarray} 
for some constant $0 <h \le 1$ independent of $\alpha$.
Therefore, 
\begin{eqnarray}\label{eq:a6}
\psi_{\alpha}(y) \ge 1-\frac{y}{h\sqrt{\alpha}}.
\end{eqnarray}
Consequently, \eqref{eq:a4} and \eqref{eq:a6} yield that for $y\in [0,h\sqrt{\alpha}]$
\begin{eqnarray}\label{eq:a7}
&& G^{\prime}(y)\ge 
\frac{\lambda}{\alpha} \frac{1}{y}\left(1-\exp\left( -\frac{y^2}{2}\right)\right)+
\frac{\lambda}{\alpha} \frac{(\psi_{\alpha}(y)-1)}{y}\left(1-\exp\left( -\frac{y^2}{2}\right)\right) \nonumber \\
&&\ge\frac{\lambda}{\alpha} \frac{1}{y}\left(1-\exp\left( -\frac{y^2}{2}\right)\right)-
\frac{\lambda}{h\alpha^{3/2}} \left(1-\exp\left( -\frac{y^2}{2}\right)\right).
\end{eqnarray} 
Integrating \eqref{eq:a7} from $0$ to $h\sqrt{\alpha},$ we obtain
\begin{eqnarray}\label{eq:a8}
G(h\sqrt{\alpha})-G(0) \ge \frac{\lambda}{\alpha} (J_1-J_2),
\end{eqnarray}
where 
\begin{eqnarray}\label{eq:a9}
J_1=\int_0^{h\sqrt{\alpha}} \left(1-\exp\left( -\frac{y^2}{2}\right)\right) \frac{dy}{y}, \qquad
J_2=\frac{1}{h \sqrt{\alpha}}\int_0^{h\sqrt{\alpha}}\left(1-\exp\left( -\frac{y^2}{2}\right)\right)dy.
\end{eqnarray}
Since $G(h \sqrt{\alpha})\le K$ and $G(0)> 0$ we have from \eqref{eq:a8} that
\begin{eqnarray}\label{eq:a10}
K \alpha\ge \lambda (J_1-J_2).
\end{eqnarray}
Next observe that the following estimates for $J_1$ and $J_2$ hold:
\begin{eqnarray}\label{eq:a11}
&& J_1= \int _1^{h\sqrt{\alpha}} \frac{dy}{y} -\int_1^{h\sqrt{\alpha}} \exp\left( -\frac{y^2}{2}\right)  \frac{dy}{y}+\int_0^1 \left(1-\exp\left( -\frac{y^2}{2}\right)\right) \frac{dy}{y} \nonumber \\
&& \ge  \int _1^{h\sqrt{\alpha}} \frac{dy}{y}-\int_1^{\infty} \exp\left( -\frac{y^2}{2}\right)  
\frac{dy}{y}=\frac{\log \alpha}{2}+\log h-c,
\end{eqnarray}
(here and below $c$ stands for a positive constant independent of $\alpha$)
and
\begin{eqnarray}\label{eq:a12}
J_2\le  \frac{1}{h \sqrt{\alpha}}\int_0^{h\sqrt{\alpha}}dy=1.
\end{eqnarray}
As a result of \eqref{eq:a11} and \eqref{eq:a12} we have that 
\begin{eqnarray}\label{eq:a13}
J_1-J_2\ge \frac{\log \alpha}{2}-c.
\end{eqnarray}
Combining \eqref{eq:a10} and \eqref{eq:a13}, we conclude that for \eqref{eq:12} to admit a solution,
we necessarily need
\begin{eqnarray}\label{eq:a14}
\lambda\le \frac{2K\alpha}{\log \alpha-c}\le
 \frac{2 K \alpha}{\log \alpha}\left(1+\frac{c}{\log{\alpha}}\right),
\end{eqnarray}
for $\alpha$ sufficiently large.

This observation immediately implies \eqref{eq:a14a},
which establishes an upper bound for $\lambda^*(\alpha)$. 
\end{proof}
To obtain a lower bound for $\lambda^*(\alpha)$ we need two technical lemmas, where the role of the parameter $w$ will be clarified later.

\begin{lemma} \label{l:4}
Let $\tilde G$ be the solution of the following linear problem:
\begin{eqnarray}\label{eq:a16}
\left\{
\begin{array}{ll}
\tilde G^{\prime\prime}+\left( \frac{1}{y}+y \varphi_{\alpha}\left(y\right)\right) \tilde G^{\prime}=
\beta_w \left(\psi_{\alpha} \left(y\right)+\frac{M_{\alpha}^2(y)}{\alpha}\right) , & 0<y<\sqrt{\alpha} \\
    \tilde G^{\prime}(0)=0, & \tilde G(\sqrt{\alpha})=K,
 \end{array}
\right.
\end{eqnarray}
with sufficiently large $\alpha$.
Here  
\begin{eqnarray} \label{eq:a15}
M_{\alpha}(t):=M\left(\frac{t}{\sqrt{\alpha}}\right),
\end{eqnarray}
with $M$  given by \eqref{eq:4},  $\varphi_{\alpha},\psi_{\alpha}$ given by \eqref{eq:7},
\begin{eqnarray}\label{eq:bw}
\beta_w=\frac{2(K-\eps_w)}{\log \alpha +c}, \qquad
\eps_w=\int_w^{\infty} \frac{ds}{f(s)},
\end{eqnarray}
where $w>0$ is an arbitrary  number and $c$ is a sufficiently large constant independent of $\alpha$.

Then, $\tilde G(y)$ is an increasing function on $y\in[0,\sqrt{\alpha}]$ and
\begin{eqnarray}
\tilde G(y) \ge \eps_w>0, \quad \mbox{on} \quad [0,\sqrt{\alpha}].
\end{eqnarray}
\end{lemma}

\begin{proof}
Using computations identical to \eqref{eq:a1a}, we rewrite \eqref{eq:a16} in the divergence form. Integrating the result from $0$ to $y$ and using the boundary condition at zero
in \eqref{eq:a16}, we get
\begin{eqnarray}\label{eq:a17}
\tilde G^{\prime} (y)=
\frac{\beta_w}{y} \int_0^y z\psi_{\alpha}(z) \exp\left( -\int_z^y s\varphi_{\alpha}(s)ds\right) dz
+\frac{\beta_w}{\alpha y} \int_0^y zM_{\alpha}^2(z) \exp\left( -\int_z^y s\varphi_{\alpha}(s)ds\right) dz
=J_3+J_4.
\end{eqnarray}
Clearly $J_3, J_4>0$ and hence $\tilde G$ is strictly increasing. Therefore, we only need to
prove the positivity of $\tilde G(0)$. 

As a first step, we establish  upper bounds on $J_3$ and $J_4$. To do so, let
\begin{eqnarray}\label{eq:a18}
F(t):=\int_0^t s \varphi_{\alpha}(s)ds, 
\end{eqnarray}
and observe that 
\begin{eqnarray}\label{eq:a19}
F(t) \le \frac{t^2}{2},
\end{eqnarray}
\begin{eqnarray}\label{eq:a20}
&& \int_0^y z\varphi_{\alpha}(z) \exp\left( -\int_z^y s\varphi_{\alpha}(s)ds\right) dz=
\int_0^y z\varphi_{\alpha}(z) \exp\left( F(z)-F(y)\right) dz \nonumber \\
&&= 1-\exp\left(-F(y) \right)\le
1-\exp\left( -\frac{y^2}{2}\right),
\end{eqnarray}
and 
\begin{eqnarray}\label{eq:a21}
\int_0^y M_{\alpha} (z) dz=\int_0^y M\left(\frac{z}{\sqrt{\alpha}} \right)dz 
=\sqrt{\alpha}\int_0^{\frac{y}{\sqrt{\alpha}}} M(s) ds 
\le \sqrt{\alpha} \int_0^1 M(s)ds=m  \sqrt{\alpha},
\end{eqnarray}
where 
\begin{eqnarray}\label{eq:a22}
m:=\int_0^1 M(s)ds.
\end{eqnarray}

Using \eqref{eq:a19} -- \eqref{eq:a22} we obtain the following upper bounds on $J_3,J_4:$
\begin{eqnarray}\label{eq:a23}
&& J_3(y)=\frac{\beta_w}{y} \int_0^y \frac{\psi_{\alpha}(z)}{\varphi_{\alpha}(z)} z \varphi_{\alpha}(z) \exp\left( -\int_z^y s\varphi_{\alpha}(s)ds\right) dz \nonumber \\
&&\le 
\frac{\beta_w}{y} M_{\alpha}(y) \int_0^y z\varphi_{\alpha}(z) \exp\left( -\int_z^y s\varphi_{\alpha}(s)ds\right) dz
 \le \frac{\beta_w}{y} M_{\alpha}(y)\left( 1-\exp\left( -\frac{y^2}{2}\right) \right),
\end{eqnarray}
and
\begin{eqnarray}\label{eq:a24}
&& J_4\le \frac{\beta_w}{\alpha y} \int_0^y zM_{\alpha}^2(z)  dz \le 
\frac{\beta_w}{\alpha} \int_0^y M_{\alpha}^2(z)  dz \le \frac{\beta_w}{\alpha} M_{\alpha}(y) \int_0^yM_{\alpha}(z)dz  \le m  \frac{\beta_w}{\sqrt{\alpha}}M_{\alpha}(y).
\end{eqnarray}
Let us also note  that
$M(t)$ is Lipshitz continuous  on the interval $[0,k]$ for some $0<k\le 1$ independent of $\alpha$
as follows from  the properties of $\varphi$ and $\psi$ and the definition of $M$.  
Therefore, when   $y\in[0,k \sqrt{\alpha}],$
\begin{eqnarray}\label{eq:a26}
M_{\alpha}(y)\le 1+ l \frac{ y}{\sqrt{\alpha}},
\end{eqnarray}
for some constant $l \ge 0$ independent of $\alpha$.

Bounds \eqref{eq:a23}, \eqref{eq:a24}, \eqref{eq:a26} and the observations presented below allow one to estimate the difference
\begin{eqnarray}\label{eq:a25}
\tilde G(\sqrt{\alpha})-\tilde G(0)= \int_0^{\sqrt{\alpha}} \tilde G^{\prime} (y)dy= \int_0^{\sqrt{\alpha}} J_3(y)dy+
\int_0^{\sqrt{\alpha}} J_4(y) dy.
\end{eqnarray}
Indeed, we have
\begin{eqnarray}\label{eq:a27}
&& \int_0^1 \frac{M_{\alpha}(y)}{y}\left( 1-\exp\left( -\frac{y^2}{2}\right) \right)dy
\le M_{\alpha}(1) \int_0^1 \left( 1-\exp\left( -\frac{y^2}{2}\right) \right)\frac{dy}{y}
\le \left(1+\frac{l}{\sqrt{\alpha}}\right)c\le  c,
\end{eqnarray}
and
\begin{eqnarray}\label{eq:a28}
&&\int_1^{\sqrt{\alpha}} \left(M_{\alpha}(y)-1\right)\frac{dy}{y} =
\int_1^{k\sqrt{\alpha}} \left(M_{\alpha}(y)-1\right)\frac{dy}{y} +
\int_{k\sqrt{\alpha}}^{\sqrt{\alpha}}  \left(M_{\alpha}(y)-1\right)\frac{dy}{y} \nonumber\\
&& \le  \frac{l}{\sqrt{\alpha}} \int_1^{k\sqrt{\alpha}}  dy + \frac{1}{k\sqrt{\alpha}} \int_{k\sqrt{\alpha}}
^{\sqrt{\alpha}} M_{\alpha}(y)dy 
 \le l k +\frac{1}{k} \int_k^1 M(s)ds\le lk +\frac{m}{k}=c.
\end{eqnarray}
Thanks to estimates \eqref{eq:a27} and \eqref{eq:a28} we  have
\begin{eqnarray}\label{eq:a29}
&& \int_0^{\sqrt{\alpha}} J_3(y) dy\le \beta_w \int_0^1 \frac{M_{\alpha}(y)}{y}\left( 1-\exp\left( -\frac{y^2}{2}\right) \right)dy+\beta_w \int_1^{\sqrt{\alpha}} \frac{dy}{y}+\beta_w\int_1^{\sqrt{\alpha}} \left(M_{\alpha}(y)-1\right)\frac{dy}{y} \nonumber \\
&& \le \beta_w \left( \frac{\log{\alpha}}{2}+c\right),
\end{eqnarray}
and
\begin{eqnarray} \label{eq:a30}
\int_0^{\sqrt{\alpha}} J_4(y) dy \le m  \frac{\beta_w}{\sqrt{\alpha}} \int_0^{\sqrt{\alpha}}M_{\alpha}(y) dy
=m \beta_w \int_0^1 M(s) ds= m^2 \beta_w.
\end{eqnarray}
Combining  \eqref{eq:a25}, \eqref{eq:a29} and \eqref{eq:a30} we obtain
\begin{eqnarray}\label{eq:a31}
\tilde G(\sqrt{\alpha})-\tilde G(0)
\le \beta_w \left( \frac{\log{\alpha}}{2}+c\right).
\end{eqnarray}
That yields
\begin{eqnarray}
\tilde G(0) \ge K-\beta_w \left( \frac{\log{\alpha}}{2}+c\right).
\end{eqnarray}
Consequently, 
\begin{eqnarray}
\tilde G(0) \ge \eps_w. 
\end{eqnarray}
In view of the monotonicity of $\tilde G$ we then have that $\tilde G \ge \eps_w$ on 
$[0,\sqrt{\alpha}]$.

\end{proof}

\begin{lemma} \label{l:5}
Let $\alpha$ be sufficiently large  and let $w>0$ be an arbitrary number. Assume that 
\begin{eqnarray}\label{eq:ll}
\lambda \le \frac{2K\alpha}{\log \alpha}\left(1-\frac{\eps_w}{K}-c \frac{f_v(w)+1}{\log \alpha}\right),
\end{eqnarray}
where $\eps_w$ is as in Lemma \ref{l:4}.
Then problem \eqref{eq:1} admits a minimal positive strictly increasing  solution satisfying $G(0) \ge \eps_w$.
\end{lemma}

\begin{proof}
We claim that $\tilde G$ constructed in Lemma \ref{l:4} is a sub-solution for problem \eqref{eq:12},
provided $\lambda$ satisfies \eqref{eq:ll}.
Indeed, 
using \eqref{eq:a17},  \eqref{eq:a23} and \eqref{eq:a24} we observe  that
\begin{eqnarray}\label{eq:a32}
\left(\tilde G^{\prime}(y)\right)^2\le c \beta_w^2 M_{\alpha}^2(y)\left( \frac{1}{\alpha}+
\left[\frac{1-\exp\left(-\frac{y^2}{2}\right)}{y}\right]^2\right).
\end{eqnarray}
In particular, this observation and \eqref{eq:a26} imply that
\begin{eqnarray}\label{eq:a33}
\left(\tilde G^{\prime}(y)\right)^2\le c \beta_w^2\le c \beta_w^2
\left(\psi_{\alpha}(y) + \frac{M_{\alpha}^2(y)}{\alpha}\right) \quad \mbox{for} \quad y\in[0,\tilde k\sqrt{\alpha}],
\end{eqnarray}
and 
\begin{eqnarray}\label{eq:a34}
\left(\tilde G^{\prime}(y)\right)^2\le c \beta_w^2  \frac{M_{\alpha}^2(y)}{\alpha} \le c \beta_w^2
\left(\psi_{\alpha}(y) + \frac{M_{\alpha}^2(y)}{\alpha}\right) \quad \mbox{for} \quad y\in[\tilde k\sqrt{\alpha},\sqrt{\alpha}],
\end{eqnarray}
where $\tilde k=\min[k, h/2]$ and $k,h$ are as in Lemmas \ref{l:4}, \ref{l:3} respectively.
Therefore,
\begin{eqnarray}\label{eq:a35}
\left(\tilde G^{\prime}(y)\right)^2 \le c \beta_w^2
\left(\psi_{\alpha}(y) + \frac{M_{\alpha}^2(y)}{\alpha}\right) \quad \mbox{for} \quad y\in[0,\sqrt{\alpha}].
\end{eqnarray}

Next we define implicitly $\tilde v(y)$ by the following formula
\begin{eqnarray}
\tilde G (y)=\int_{\tilde v(y)}^{\infty} \frac{ds}{f(s)},
\end{eqnarray} 
i.e. $\tilde v(y)=G^{-1}(\tilde G(y))$ where $G^{-1}:(0,K]\to[0,\infty)$ is implicitly defined by \eqref{eq:8}.
 Since $G^{-1}$ is decreasing and $\tilde G(y)$ is an increasing function of $y$ we have that $\tilde v(y)$ is a decreasing function
 of $y$. In view of this observation and an assumption that  $f$ is convex and increasing  we have
\begin{eqnarray}\label{eq:a37}
f_v(\tilde v(y) )\le f_v(w) \quad \mbox{for} \quad y\in[0,\sqrt{\alpha}].
\end{eqnarray}
Therefore, $\tilde G$ will be a sub-solution for \eqref{eq:12}, provided that
\begin{eqnarray}\label{eq:a38}
\tilde G^{\prime\prime}+\left( \frac{1}{y}+y \varphi_{\alpha}\left(y\right)\right) \tilde G^{\prime}\ge
\frac{\lambda}{\alpha}\psi_{\alpha} \left(y\right) +\left( \tilde G^{\prime}\right)^2 f_v(w), & 0<y<\sqrt{\alpha}.
\end{eqnarray}
Using \eqref{eq:a16} and \eqref{eq:a35}, we observe that this condition is automatically satisfied
if
\begin{eqnarray}\label{eq:a39}
\beta_w \left(\psi_{\alpha} \left(y\right)+\frac{M_{\alpha}^2(y)}{\alpha}\right) \ge \frac{\lambda}{\alpha} \psi_{\alpha}(y) +c \beta_w^2 f_v(w) \left(\psi_{\alpha}(y) + \frac{M_{\alpha}^2(y)}{\alpha}\right)
\quad 0<y<\sqrt{\alpha}, 
\end{eqnarray}
which  is in turn satisfied when
\begin{eqnarray}\label{eq:a40}
\beta_w(1-c\beta_w f_v(w))\ge \frac{\lambda}{\alpha}.
\end{eqnarray}
Straightforward computations show that \eqref{eq:a40} holds for all values of 
$\lambda$ satisfying \eqref{eq:ll}.
Consequently, for $\lambda$ satisfying \eqref{eq:ll}, problem \eqref{eq:12} admits a positive  sub-solution
and thus by Corollaries \ref{l:1}, \ref{l:2} problem \eqref{eq:6} and hence problem \eqref{eq:1} admits a minimal solution.
\end{proof}

An immediate consequence of this lemma is the following corollary.
\begin{corollary}\label{c:1}
\begin{eqnarray}\label{eq:lb}
\lambda^*(\alpha) \ge \sup_{w\in(0,\infty)}\frac{2K\alpha}{\log \alpha}\left(1-\frac{\eps_w}{K}-c \frac{f_v(w)+1}{\log \alpha}\right).
\end{eqnarray}
where $w$ and $\eps_w$ are as in Lemma \ref{l:4}. 
 \end{corollary}

We now can proceed to the proof of Theorem \ref{t:1}.

\begin{proof}[Proof of Theorem \ref{t:1}]
The statement of the theorem follows directly form Lemma \ref{l:3} and Corollary \ref{c:1}.
First we observe that by \eqref{eq:a14a}
\begin{eqnarray}\label{eq:a50}
\limsup_{\alpha\to\infty} \frac{ \lambda^*(\alpha)\log \alpha}{\alpha} \le 2K.
\end{eqnarray} 
On the other hand
given $\eps>0$ arbitrary small, we can choose $w$ sufficiently large so that  $\eps_w/K<\eps/2$.
This observation together with \eqref{eq:lb} gives
\begin{eqnarray}\label{eq:a51}
\liminf_{\alpha\to\infty} \frac{ \lambda^*(\alpha)\log \alpha}{\alpha} \ge 2(K-\eps).
\end{eqnarray} 
In view of the fact that $\eps$ is arbitrarily small, \eqref{eq:a50} and \eqref{eq:a51} give
\eqref{eq:a1}.
\end{proof}

\section{ Asymptotic behavior of the extremal solution as $\alpha\to\infty.$}

In this section we give a proof of Theorem \ref{t:2}. For convenience we split the
section into two parts dealing with local and integral properties of the extremal solution respectively, that
is, with parts one and two of Theorem \ref{t:2}. In this section we will use \eqref{eq:6} as 
an alternative version of \eqref{eq:1}.

\subsection{Local properties of the extremal solution}
In this subsection we give a proof of the first part of Theorem \ref{t:2}. The proof is based on the
following two lemmas.

\begin{lemma}\label{l:10}
Assume that
 \begin{eqnarray}\label{eq:d0}
  \vartheta <\frac{\lambda}{\lambda^*}\le1 \quad \mbox{with} \quad 0<\vartheta \le 1
 \end{eqnarray} 
 Then, any radial solution of \eqref{eq:1} satisfies for any $\frac{1}{4}>\delta>0$
the following upper bound:
\begin{eqnarray}\label{eq:d1}
u (x)\le c\left(\frac{1+\log(1/|x|)}{\log {\alpha}} \right), \quad  \quad |x|\in\left [\frac{1}{\alpha^{\frac12-2\delta}},1\right],
\end{eqnarray}
where constant $c=c(\vartheta,\delta)>0$ is independent of $\alpha.$ 
\end{lemma}
\begin{proof}
First we observe that \eqref{eq:6} can be rewritten in the divergence form
\begin{eqnarray}\label{eq:d2}
-\left( y\exp\left(\int_0^y s\varphi_{\alpha}(s)ds\right) v^{\prime}(y)\right)^{\prime}=\frac{\lambda}{\alpha}
y \exp\left(\int_0^y s\varphi_{\alpha}(s)ds\right) \psi_{\alpha}(y) f(v(y)). 
\end{eqnarray}
Integrating this equation and taking into account the boundary condition at zero in \eqref{eq:6},
we get
\begin{eqnarray}\label{eq:d3}
-y v^{\prime}(y)=\frac{\lambda}{\alpha} I(y),
\end{eqnarray}
where 
\begin{eqnarray}\label{eq:d4}
I(y)=\int_0^y f(v(z)) z \psi_{\alpha}(z) \exp\left(-\int_z^y s\varphi_{\alpha}(s)ds \right)dz=
\exp(-F(y)) \int_0^y f(v(z)) z \psi_{\alpha}(z) \exp\left( F(z) \right) dz,
\end{eqnarray}
and $F(t)$ is defined in \eqref{eq:a18}.

We  claim that
\begin{eqnarray}\label{eq:d5}
I(y_2) \le c M_{\alpha}(y_2) I(y_1) \quad \mbox{for} \quad y_2>y_1, \quad y_1\in [1,q\sqrt{\alpha}], 
\end{eqnarray}
with  some $0<q\le1$ independent of $\alpha.$

To prove this claim we observe that
\begin{eqnarray}\label{eq:d6}
I(y_2)= \exp\left(F(y_1)-F(y_2)\right) I(y_1)+\exp\left(-F(y_2)\right) \int_{y_1}^{y_2} f(v(z)) \frac{\psi_{\alpha}(z)}{\varphi_{\alpha}(z)} z \varphi_{\alpha}(z) \exp\left(F(z) \right) dz. 
\end{eqnarray}
Since $F(t)$ and $M_{\alpha}(t)\ge 1$ are increasing and $f(v(t))$ is decreasing
on $[0,\sqrt{\alpha}],$ using calculations similar to  \eqref{eq:a20}, we obtain
\begin{eqnarray}\label{eq:d7}
&& I(y_2)\le  I(y_1)+\exp\left(-F(y_2)\right) f(v(y_1)) M_{\alpha}(y_2)  \int_{y_1}^{y_2}  z \varphi_{\alpha}(z) \exp\left(F(z) \right) dz \le \nonumber \\
&& I(y_1)+f(v(y_1)) M_{\alpha}(y_2) \left(1-\exp \left(F(y_2)-F(y_1)\right) \right) \le I(y_1)+f(v(y_1)) M_{\alpha}(y_2).
\end{eqnarray}
Moreover, since $\varphi,\psi$ are Lipshitz continuous and $\varphi(0)=\psi(0)=1$ we can choose $q>0$ independent
of $\alpha$ such that
\begin{eqnarray}\label{eq:mover}
\frac{\psi_{\alpha}(y)}{\varphi_{\alpha}(y)}\ge \frac12, \quad \varphi_{\alpha}(y)\ge \frac 12 \quad\mbox{for}\quad y\le q\sqrt{\alpha}.
\end{eqnarray}
Therefore, as follows from \eqref{eq:mover} and \eqref{eq:d4}
\begin{eqnarray}\label{eq:d8}
&& I(y_1)\ge f((v(y_1)) \exp(-F(y_1))\int_0^{y_1} z \psi_{\alpha}(z) \exp(F(z))dz \ge \nonumber \\
&& \frac12  f((v(y_1)) \exp(-F(y_1))\int_0^{y_1} z \phi_{\alpha}(z) \exp(F(z))dz=
\frac12 f(v(y_1)) \left(1-\exp(-F(y_1))\right)\ge c f(v(y_1)),
\end{eqnarray}
provided $y_1<q\sqrt{\alpha}$.
Combining \eqref{eq:d7} and \eqref{eq:d8}, we get \eqref{eq:d5}.

Now we fix $\delta>0$ sufficiently small. We claim that  for large enough $\alpha,$
\begin{eqnarray}\label{eq:d9}
v(\alpha^{\delta})\le c,
\end{eqnarray}
where $c$ depends only  on $\vartheta$ and $\delta.$
To show that, we integrate \eqref{eq:a7}  from zero to $\alpha^{\delta}$ and arguing as in Lemma
\ref{l:3} obtain 
\begin{eqnarray}\label{eq:d10}
G(\alpha^{\delta})\ge \frac{\lambda}{\alpha}\left( \delta \log \alpha-c\right).
\end{eqnarray}
Since  for $\alpha$ large enough, 
\begin{eqnarray}\label{eq:d11}
\lambda^*\ge \frac{K\alpha}{\log \alpha},
\end{eqnarray}
we have that for $\lambda$ satisfying \eqref{eq:d0}
\begin{eqnarray}\label{eq:d12}
G(\alpha^{\delta}) \ge \frac{\vartheta \delta K}{2}.
\end{eqnarray}
Using definition of G (see \eqref{eq:8}), we have
\begin{eqnarray}\label{eq:d13}
\frac{1}{K} \int_{v(\alpha^{\delta})}^\infty \frac{d s}{f(s)} \ge \frac{\vartheta \delta}{2}>0.
\end{eqnarray}
Since the integral in \eqref{eq:d13} is bounded from below away from zero independently of $\alpha,$ we conclude
that the lower limit in this integral is bounded from above, which gives \eqref{eq:d9}.

Next, taking $y_2=\alpha^{2\delta},$ we have from \eqref{eq:d5} that
\begin{eqnarray}\label{eq:d14}
I(\alpha^{2\delta} )\le c M_{\alpha}(\alpha^{2\delta}) I(y), \qquad \alpha^{\delta}\le y \le \alpha^{2\delta}.
\end{eqnarray}

Since $1\le M_{\alpha}(\alpha^{2\delta}) <c,$ we have
\begin{eqnarray}\label{eq:d15}
I(y) \ge c I(\alpha^{2\delta}), \qquad \alpha^{\delta}\le y \le \alpha^{2\delta},
\end{eqnarray}
and hence by \eqref{eq:d3}
\begin{eqnarray}\label{eq:d16}
-yv^{\prime}(y) \ge -c \alpha^{2\delta} v^{\prime}(\alpha^{2\delta})=cD,
\end{eqnarray}
where
\begin{eqnarray}\label{eq:d16a}
D=- \alpha^{2\delta} v^{\prime}(\alpha^{2\delta}).
\end{eqnarray}
Integrating  inequality \eqref{eq:d16}, we obtain
\begin{eqnarray}\label{eq:d17}
v(\alpha^{\delta})-v(\alpha^{2\delta})= -\int_{\alpha^{\delta}}^{\alpha^{2\delta}} v^{\prime}(y)dy
\ge D\int_{\alpha^{\delta}}^{\alpha^{2\delta}} \frac{dy}{y}=c \delta D \log \alpha.
\end{eqnarray}
Therefore, from \eqref{eq:d9} and \eqref{eq:d17} we have
\begin{eqnarray}\label{eq:d18}
D\le \frac{c}{\log \alpha}.
\end{eqnarray}
On the other hand, from \eqref{eq:d5} with $y_1=\alpha^{2\delta}$ and $y_2=y$ we have
\begin{eqnarray}\label{eq:d19}
I(y)\le c M_{\alpha}(y) I(\alpha^{2\delta}), \quad y\ge \alpha^{2\delta},
\end{eqnarray}
and, therefore, by \eqref{eq:d3}, \eqref{eq:d16a}, \eqref{eq:d18} and \eqref{eq:d19} 
\begin{eqnarray}\label{eq:d20}
-yv^{\prime}(y) \le c M_{\alpha}(y) \left( -\alpha^{2\delta} v^{\prime}(\alpha^{2\delta})\right)=
c D M_{\alpha}(y) \le \frac{c}{\log \alpha} M_{\alpha}(y).
\end{eqnarray}
Thus, for $y\ge \alpha^{{2\delta}},$ we have
\begin{eqnarray}\label{eq:d21}
-v^{\prime}(y)\le \frac{c}{\log \alpha} \frac{M_{\alpha}(y)}{y}.
\end{eqnarray}
Integrating \eqref{eq:d21}, we then obtain
\begin{eqnarray}\label{eq:d22}
v(y)=-\int_y^{\sqrt{\alpha}} v^{\prime}(y)dy \le \frac{c}{\log \alpha} \int_y^{\sqrt{\alpha}} M_{\alpha}(y)\frac{dy}{y}.
\end{eqnarray}
Now using \eqref{eq:a15},\eqref{eq:a21}, \eqref{eq:a22} and \eqref{eq:a26} we
estimate the right hand side of \eqref{eq:d22}. First assume that $y\ge k \sqrt{\alpha},$ 
where $k$ is as in Lemma \ref{l:4}. Then
\begin{eqnarray}\label{eq:d23}
\int_y^{\sqrt{\alpha}} M_{\alpha}(y) \frac{dy}{y} \le \int_{k\sqrt{\alpha}}^{\alpha} M_{\alpha}(y) \frac{dy}{y} 
\le \frac{1}{k\sqrt{\alpha}} \int_{k\sqrt{\alpha}}^{\sqrt{\alpha}} M_{\alpha}(y)dy\le
\frac{1}{k} \int_k^1 M(t)dt\le \frac{m}{k}<c.
\end{eqnarray}
Next, when $y<k\sqrt{\alpha},$ we have
\begin{eqnarray}\label{eq:d24}
&&\int_y^{\sqrt{\alpha}} M_{\alpha}(y) \frac{dy}{y}=\int_{k\sqrt{\alpha}}^{\sqrt{\alpha}} M_{\alpha}(y) \frac{dy}{y}+\int_y^{k\sqrt{\alpha}} M_{\alpha}(y) \frac{dy}{y}\le c+
\int_y^{k\sqrt{\alpha}} \left(1+l \frac{y}{\sqrt{\alpha}}\right) \frac{dy}{y} \le \nonumber \\
&& c+\int_y^{k\sqrt{\alpha}} \frac{dy}{y}+\frac{l}{\sqrt{\alpha}} \int_y^{k\sqrt{\alpha}} dy
= c+\log(k \sqrt{\alpha})-\log(y)+lk -l\frac{y}{\sqrt{\alpha}} \le c +\log\left( \frac{\sqrt{\alpha}}{y}\right).
\end{eqnarray}
Combining \eqref{eq:d22}, \eqref{eq:d23} and \eqref{eq:d24}, we get
\begin{eqnarray}\label{eq:x10}
v(y) \le \frac{c}{\log \alpha} \left( 1+\log \left( \frac{\sqrt{\alpha}}{y}\right)\right), \quad y\ge \alpha^{2\delta}.
\end{eqnarray}
The latter inequality in terms of the original (unscaled) variables (see \eqref{eq:5}) gives \eqref{eq:d1},
which completes the proof.
\end{proof}

\begin{remark} We note that the statement of the lemma above concerns not only extremal but 
all radial solutions of problem \eqref{eq:1}. That is, any radial solution of \eqref{eq:1} with $\lambda$ 
comparable with $\lambda^*$ obeys \eqref{eq:d1}. This fact, in particular, implies that any
radial solution of \eqref{eq:1} with $\lambda\asymp\lambda^*$ tends to zero outside of the
origin as $\alpha\to\infty$. 
\end{remark}

\begin{lemma} \label{l:10a}
The extremal solution $u_{\alpha}^*$ of \eqref{eq:1} satisfies
\begin{eqnarray} 
u_{\alpha}^* (0) \to \infty \quad \mbox{as} \quad \alpha\to \infty.
\end{eqnarray}
\end{lemma}

\begin{proof}  First we observe that due to the monotonicity of $G^{\prime},$ inequality \eqref{eq:a7}, and arguments identical to these given in Lemma \ref{l:3},
we have
\begin{eqnarray}\label{eq:pp1}
G(\sqrt{\alpha})-G(\tilde c)=\int_{\tilde c}^{\sqrt{\alpha}} G^{\prime}(y)dy\ge \int_{\tilde c}^{h\sqrt{\alpha}} G^{\prime}(y)dy
\ge \frac{\lambda}{\alpha}\frac{\log \alpha}{2}\left(1-\frac{c}{\log \alpha}\right),
\end{eqnarray}
where $\tilde c\ge 0$ is arbitrary constant independent of $\alpha.$
Also we observe that Theorem \ref{t:1} implies
\begin{eqnarray}\label{eq:pp2}
\lambda^*(\alpha)=\frac{2K\alpha}{\log \alpha}\left(1-\sigma(\alpha)\right),
\end{eqnarray}
with
\begin{eqnarray}
\sigma(\alpha)\to 0, \quad \mbox{as} \quad \alpha \to \infty.
\end{eqnarray}

Since $G(\sqrt{\alpha})=K$ \eqref{eq:pp1} and \eqref{eq:pp2} give that for $\lambda=\lambda^*$
the following estimate holds:
\begin{eqnarray}\label{eq2:28}
G_{\alpha}^*(\tilde c)\le K-\frac{\lambda^*(\alpha) \log \alpha}{2\alpha}\left(1-\frac{c}{\log \alpha}\right)\le c\left(\frac{1}{\log \alpha }+\sigma(\alpha)\right),
\end{eqnarray}
where $G_{\alpha}^*(y)=G(v_{\alpha}^*(y))$.

Thus, $G_{\alpha}^*(\tilde c)\to 0$ as $\alpha\to\infty$.
Therefore,
\begin{eqnarray}
\int_{v_{\alpha}^*(\tilde c)}^{\infty} \frac{ds}{f(s)}\to 0.
\end{eqnarray}
Hence $v_{\alpha}^*(\tilde c)\to \infty$  as $\alpha\to\infty$ and,
consequently,  $v_{\alpha}^*(0)\to \infty$  as $\alpha\to\infty$.

Since $u_{\alpha}^*(0)=v^*_{\alpha}(0)$
we conclude that $u_{\alpha}^*(0)\to \infty$ as $\alpha\to\infty$.
\end{proof}

We now can give a proof of the first part of Theorem \ref{t:2}

\begin{proof}[Proof of Theorem \ref{t:2} part 1]
By Lemma \ref{l:10} we have that
\begin{eqnarray} \label{eq:pp5}
u^*_{\alpha}(x)\le c \frac{\log(\log \alpha)}{\log\alpha}, \qquad |x|\ge \frac{1}{\log{\alpha}}.
\end{eqnarray}
Taking a limit as $\alpha\to \infty$  in \eqref{eq:pp5}, we obtain that 
$u^*_{\alpha} (x)\to 0$ for $x\ne 0$ as $\alpha\to \infty$.
The fact that  $u_{\alpha}^*(0)\to\infty$ as   $\alpha\to \infty$ follows directly from Lemma \ref{l:10a}.
\end{proof}

We now proceed to the proof of the second part of Theorem \ref{t:2}.

\subsection{Integral properties of the extremal solution}

In this section we complete the proof of Theorem \ref{t:2},  which follows from the following two lemmas.

\begin{lemma}\label{l:theta}
Let $\theta^*(\alpha)$ be the largest solution of the equation
\begin{eqnarray}\label{eq:th1}
\frac{f(\theta)}{\theta}=c \log \alpha,
\end{eqnarray}
where $c>0$ is a fixed constant.

Then, for arbitrarily  small $\gamma>0,$ we have
\begin{eqnarray}\label{eq:th2}
\theta^*(\alpha) \le \alpha^{\gamma},
\end{eqnarray}
provided $\alpha$ is sufficiently large.
\end{lemma}
\begin{proof}
Let $\theta>1$. Then, by the convexity of $f,$ we have
\begin{eqnarray}\label{eq:th3}
f(s)\le g(s)=f(\sqrt{\theta})+\frac{f(\theta)-f(\sqrt{\theta})}{\theta-\sqrt{\theta}}(s-\sqrt{\theta}), \quad s\in[\sqrt{\theta},\theta],
\end{eqnarray}
and
\begin{eqnarray}\label{eq:th4}
f(t\theta)\le tf({\theta})+(1-t)f(0), \quad t\in[0,1].
\end{eqnarray}
In particular, setting $t=1/\sqrt{\theta},$ the latter inequality gives
\begin{eqnarray}\label{eq:th5}
f(\sqrt{\theta})\le \frac{f(\theta)}{\sqrt{\theta}}+c.
\end{eqnarray}
Next let
\begin{eqnarray}\label{eq:th6}
\rho(\theta)=\int_{\sqrt{\theta}}^{\theta} \frac{ds}{f(s)}.
\end{eqnarray}
By \eqref{eq:th3} we have
\begin{eqnarray}\label{eq:th7}
\rho(\theta)\ge \int_{\sqrt{\theta}}^{\theta} \frac{ds}{g(s)}
=\frac{\theta-\sqrt{\theta}}{f(\theta)-f(\sqrt{\theta})}\log\left( \frac{f(\theta)}{f(\sqrt{\theta})}\right).
\end{eqnarray}
This observation together with \eqref{eq:th5} implies that for $\theta$ sufficiently large
we have
\begin{eqnarray}
\rho(\theta)> c \frac{\theta \log(\theta)}{f(\theta)}.
\end{eqnarray}
By \eqref{eq:2}, $\rho(\theta)\to 0$ as $\theta\to \infty$ and therefore,
\begin{eqnarray}
\frac{f(\theta)}{\theta} > \log(\theta) \chi(\theta),
\end{eqnarray}
with some function $\chi(\theta)$ having the property that $\chi(\theta)\to \infty$ as $\theta\to \infty$.
In view of this observation we have
\begin{eqnarray}
\log(\theta^*) \chi(\theta^*) \le c \log \alpha.
\end{eqnarray}
The statement of the lemma then follows immediately.
\end{proof}

\begin{lemma}\label{l:x} 
Let $\delta,\gamma>0$ be arbitrary fixed small numbers such that $\gamma+4\delta<1$.	
If $\alpha$ is large enough, then 
there exists a point $a<\alpha^{2\delta}$  such that
\begin{eqnarray}\label{eq:x6}
\int_0^a v_{\alpha}^*(y)ydy \le c \alpha^{\gamma+4\delta}, \qquad \int_0^a \psi_{\alpha}(y)f(v_{\alpha}^*(y))ydy \le c \alpha^{\gamma+4\delta}\log\alpha,
\end{eqnarray}
and 
\begin{eqnarray}\label{eq:x7}
v_{\alpha}^*(y)\le c \alpha^{\gamma} , \quad f(v_{\alpha}^*(y)) \le c\alpha^{\gamma}\log \alpha
\quad \mbox{when} \quad y\ge a,
\end{eqnarray}
where $c=c(\delta,\gamma)>0$ is a constant independent of $\alpha$.
\end{lemma}

\begin{proof}
First we claim that 
\begin{eqnarray}
v_{\alpha}^*(\alpha^{\delta})>c.
\end{eqnarray}
Indeed, arguing as in Lemma \ref{l:10} (see Eqs. \eqref{eq:d2}, \eqref{eq:d3}, \eqref{eq:d4}),
we have that
\begin{eqnarray}
-v^{\prime}(y)=\frac{\lambda}{\alpha}\frac{\exp(-F(y))}{y}\int_0^y\frac{\psi_{\alpha}(z)}{\varphi_{\alpha}(z)}
z\varphi_{\alpha}(z) f(v(z)) \exp(F(z)) dz,
\end{eqnarray}
where $F$ is defined by \eqref{eq:a18}.

Next, using \eqref{eq:mover},  \eqref{eq:a20},
and the fact that $f(v(y))\ge f(0)>0,$ we have that for $\alpha^{2\delta}\le y\le q\sqrt{\alpha}$
\begin{eqnarray}
&&-v^{\prime}(y)\ge c \frac{\lambda}{\alpha} \frac{\exp( -F(y))}{y}\int_0^y z\varphi_{\alpha}(z)\exp(F(z))dz=\nonumber \\
&&c\frac{\lambda}{\alpha} \frac{1}{y}(1-\exp(-F(y))\ge c\frac{\lambda}{\alpha} \frac{1}{y}(1-\exp(-c\alpha^{4\delta})\ge  c\frac{\lambda}{\alpha} \frac{1}{y}.
\end{eqnarray}
 Integrating this expression from  $\alpha^{2\delta}$ to $q\sqrt{\alpha}$ we get
 \begin{eqnarray}
v(\alpha^{2\delta})\ge v(q\sqrt{\alpha})+c\frac{\lambda}{\alpha}(\frac12-2\delta)(\log \alpha-c).
\end{eqnarray}
 For $\delta<\frac14,$ the latter inequality implies  that for sufficiently large $\alpha$
\begin{eqnarray}
v(\alpha^{2\delta})\ge c\frac{\lambda}{\alpha}\log \alpha.
\end{eqnarray}
In particular, we have
\begin{eqnarray}\label{eq:x1}
v_{\alpha}^*(\alpha^{2\delta})\ge c \frac{\lambda^* \log \alpha}{\alpha}\ge c,
\end{eqnarray}
which proves our claim.

Next let
\begin{eqnarray}
\Gamma(y)=y [v_{\alpha}^*(y)]^{\prime}+y^2 \varphi_{\alpha}(y) v_{\alpha}^*(y).
\end{eqnarray}
We note that since $\varphi$ is Lipshitz continuous $\varphi^{\prime}$ is defined almost everywhere
and $|\varphi^{\prime}(y)|<c.$
Direct computations give
\begin{eqnarray}\label{eq:x2}
\Gamma^{\prime}(y)=\left[ \left( 2\varphi_{\alpha}(y)+\frac{\varphi_{}^{\prime}(y)}{\sqrt{\alpha}}y\right) v_{\alpha}^*(y)-
\frac{\lambda^*}{\alpha} \psi_{\alpha}(y) f(v_{\alpha}^*(y))\right] y
\end{eqnarray}
for almost every $y.$

Clearly, we have $\Gamma(0)=\Gamma^{\prime}(0)=0$. Assume first that $v_{\alpha}^*(0)\ge v_0$
where $v_0$ is the largest solution of
\begin{eqnarray}
4v_0=\frac{\lambda^*}{\alpha} f(v_0).
\end{eqnarray}
Then $\Gamma(y)$ is negative in some small neighborhood of $y=0.$  
On the other hand $\Gamma(\alpha^{2\delta})>0$ as follows from \eqref{eq:d21} and \eqref{eq:x1}.
Consequently, there exists a point $y=a\in(0,\alpha^{2\delta})$ such that $\Gamma(a)=0$.
This in particular implies that there exists a point $0<a_0<a$ 
where $\Gamma$ attains its minimum. At that point we have
\begin{eqnarray}
\left( 2\varphi_{\alpha}(a_0)+O\left(\frac{a_0}{\sqrt{\alpha}}\right)\right) v_{\alpha}^*(a_0)=
\frac{\lambda^*}{\alpha} \psi_{\alpha}(a_0) f(v_{\alpha}^*(a_0)).
\end{eqnarray}
This implies that
\begin{eqnarray}
c v_{\alpha}^*( a_0)=\frac{\lambda^*}{\alpha} f(v_{\alpha}^*(a_0)).
\end{eqnarray}
Therefore, as follows form Lemma \ref{l:theta} and the monotonicity of $v_{\alpha}^*(y),$ 
for $y\ge a_0$ we have
\begin{eqnarray}
v_{\alpha}^*(y)< c \alpha^{\gamma}, \qquad f(v_{\alpha}^*(y)) \le c \alpha^{\gamma}\log \alpha.
\end{eqnarray}
Taking into account that $a_0<a$ and the monotonicity of $v_{\alpha}^*(y),$ the latter inequalities
give \eqref{eq:x7}.

Next integrating \eqref{eq:x2} we have 
\begin{eqnarray}
\int_0^a \Gamma^{\prime}(y)dy=\int_0^a \left[ \left( 2\varphi_{\alpha}(y)+\frac{\varphi_{}^{\prime}(y)}{\sqrt{\alpha}}y\right) v_{\alpha}^*(y)-
\frac{\lambda}{\alpha} \psi_{\alpha}(y) f(v_{\alpha}^*(y))\right] ydy =0.
\end{eqnarray}
Therefore,
\begin{eqnarray}\label{eq:x3}
\int_0^a \left( 2\varphi_{\alpha}(y)+\frac{\varphi_{}^{\prime}(y)}{\sqrt{\alpha}}y\right) v_{\alpha}^*(y)ydy=
\int_0^a
\frac{\lambda}{\alpha} \psi_{\alpha}(y) f(v_{\alpha}^*(y))ydy. 
\end{eqnarray}
This observation and Jensen inequality imply that
\begin{eqnarray}
\langle v \rangle\ge c \frac{\lambda^*}{\alpha} f(\langle v\rangle),
\end{eqnarray}
where 
\begin{eqnarray}
\langle v \rangle=\frac{2}{a^2} \int_0^a v_{\alpha}^*(y)ydy,
\end{eqnarray}
is an average of $v$ over $[0,a]$.
Consequently,  by Lemma \ref{l:theta} we have
\begin{eqnarray}
\langle v\rangle \le c \alpha^{\gamma},
\end{eqnarray}
and thus
\begin{eqnarray}\label{eq:x4}
\int_0^a v_{\alpha}^*(y)ydy < c \alpha^{\gamma+4\delta}.
\end{eqnarray}
The latter inequality and \eqref{eq:x3} imply that
\begin{eqnarray}\label{eq:x5}
\int_0^a \psi_{\alpha}(y)f(v_{\alpha}^*(y))ydy < c  \alpha^{\gamma+4\delta}\log \alpha.
\end{eqnarray}
This proves \eqref{eq:x6}.

Assume now that  $v_{\alpha}^*(0)\le v_0.$ In this case by Lemma \ref{l:theta}
\begin{eqnarray} 
v^*_{\alpha}(y)\le \alpha^{\gamma}, \quad f(v_{\alpha}^*(y))\le c \alpha^{\gamma}\log \alpha 
\quad \mbox{on} \quad [0,\sqrt{\alpha}].
\end{eqnarray}
which proves \eqref{eq:x7}. Using these inequalities and taking $a=\alpha^{2\delta}$ we also
have \eqref{eq:x6} for this case.
\end{proof}

We now can proceed to the proof of the second part of Theorem \ref{t:2}.

\begin{proof}[Proof of Theorem \ref{t:2} part 2]
We first observe that
\begin{eqnarray}
&& \int_0^{\sqrt{\alpha}} v_{\alpha}^*(y)ydy=\left\{\int_0^a +\int_a^{\alpha^{2\delta}}  
+\int_{\alpha^{2\delta}}^{\frac{\sqrt{\alpha}}{\log \alpha}} +\int_{\frac{\sqrt{\alpha}}{\log \alpha}}^{\sqrt{\alpha}} \right\}v_{\alpha}^*(y)ydy=L_1+L_2+L_3+L_4.
\end{eqnarray}
By \eqref{eq:x6}, \eqref{eq:x7} (see Lemma \ref{l:x}) we have that
\begin{eqnarray}
L_1,L_2 \le c  \alpha^{\gamma+4\delta}.
\end{eqnarray}
Moreover by \eqref{eq:x10} (see Lemma \ref{l:10}) we also have that
\begin{eqnarray}
L_3\le c \frac{\alpha}{(\log \alpha)^2}, \quad L_4\le c\frac{\log(\log \alpha)}{\log \alpha}\alpha.
\end{eqnarray}
Hence,
\begin{eqnarray}
\int_0^{\sqrt{\alpha}} v_{\alpha}^*(y)ydy\le c \frac{\log(\log \alpha)}{\log \alpha}\alpha.
\end{eqnarray}
Observing that 
\begin{eqnarray}\label{eq:xx1}
\int_{B} u_{\alpha}^*(x)dx =\frac{2\pi}{\alpha} \int_0^{\sqrt{\alpha}} v_{\alpha}^*(y)ydy \le c \frac{\log(\log \alpha)}{\log \alpha}.
\end{eqnarray}
Taking a limit  as $\alpha\to \infty$ in the right hand side of \eqref{eq:xx1} and using the positivity of $u_{\alpha}^*,$ we obtain the first part of \eqref{eq:t222}.

Next, we perform a computation similar to those above, 
\begin{eqnarray}
&& \int_0^{\sqrt{\alpha}}\psi_{\alpha}(y) f(v_{\alpha}^*(y))ydy=\left\{\int_0^a +\int_a^{\alpha^{2\delta}}  
+\int_{\alpha^{2\delta}}^{\frac{\sqrt{\alpha}}{\log \alpha}} +\int_{\frac{\sqrt{\alpha}}{\log \alpha}}^{\sqrt{\alpha}} \right\}\psi_{\alpha}(y) f(v_{\alpha}^*(y))ydy=\nonumber \\
&&P_1+P_2+P_3+P_4.
\end{eqnarray}
By \eqref{eq:x6}, \eqref{eq:x7} (see Lemma \ref{l:x}) we have that
\begin{eqnarray}
P_1,P_2 \le c  \alpha^{\gamma+4\delta}\log \alpha.
\end{eqnarray}
By \eqref{eq:d9} we have
\begin{eqnarray}
P_3\le c \frac{\alpha}{(\log \alpha)^2},
\end{eqnarray}
and  by \eqref{eq:x10}
\begin{eqnarray}
P_4\le f\left ( \frac{\log(\log \alpha)}{\log \alpha}\right) \int_{\frac{\sqrt{\alpha}}{\log \alpha}}^{\sqrt{\alpha} }\psi_{\alpha}(y) ydy,
\end{eqnarray}

Arguing as above we then have
\begin{eqnarray}
&&\int_{B} \psi(x) f(u_{\alpha}^*(x))dx=\frac{2\pi}{\alpha} \int_0^{\sqrt{\alpha}}\psi_{\alpha}(y) f(v_{\alpha}^*(y))ydy 
\le \frac{c}{(\log\alpha)^2}+\frac{2\pi}{\alpha}  f\left ( \frac{\log(\log \alpha)}{\log \alpha}\right) \int_{\frac{\sqrt{\alpha}}{\log \alpha}}^{\sqrt{\alpha} }\psi_{\alpha}(y) ydy \nonumber\\
&&=\frac{c}{(\log\alpha )^2}+  f\left ( \frac{\log(\log \alpha)}{\log \alpha}\right) 
\int_{B\setminus B\left(0,\frac{1}{\log(\alpha)}\right)} \psi(x)dx.
\end{eqnarray}
Therefore, 
\begin{eqnarray}
\int_{B} \psi(x) f(u_{\alpha}^*(x))dx\le f(0)\int_{B} \psi(x) dx+\tilde \sigma(\alpha),
\end{eqnarray}
for some $\tilde \sigma(\alpha)$ having the property that $ \tilde \sigma(\alpha) \to 0$ as
$\alpha \to \infty$.  In view of this observation and the fact that
\begin{eqnarray}
\int_{B} \psi(x) f(u_{\alpha}^*(x))dx> f(0)\int_{B} \psi(x) dx,
\end{eqnarray}
which follows from the positivity of  $u_{\alpha}^*,$ we have the second part of \eqref{eq:t222}, which completes the proof.

\end{proof}

We now turn to the proof of Theorem \ref{t:3}.

\section{Proof of Theorem \ref{t:3}}

The proof of Theorem \ref{t:3} requires  the following lemma. 
This lemma is based on a rescaled  version of inequality  \eqref{eq:nedev}
which was first introduced in \cite{CR75}.

\begin{lemma}\label{l:rest}
 Let $v_{\alpha}^*$ be an extremal solution of \eqref{eq:6} and set
\begin{eqnarray}\label{eq1:01}
f^{\sharp}(t)=f(v^*_{\alpha}(1)+t), \quad f_v^{\sharp}(t)=f_v(v^*_{\alpha}(1)+t), \quad
 \tilde f^{\sharp}(t)=f^{\sharp}(t)-f^{\sharp}(0),
\quad g^{\sharp}(t)=\int_0^t (f_v^{\sharp}(s))^2 ds.
\end{eqnarray}

Assume that there exist constants  $0<\tilde c_0<1$ and  $\tilde c_1>1$ such that for sufficiently
large $\alpha$ 
\begin{eqnarray}\label{eq1:02}
f^{\sharp}(t) \ge \tilde c_1 f^{\sharp}(0), \quad t \ge 0,
\end{eqnarray}
implies 
\begin{eqnarray}\label{eq1:03}
(1-\tilde c_0) \tilde f^{\sharp}(t)  f^{\sharp}_v(t) \ge g^{\sharp}(t),
\end{eqnarray}

Then, for sufficiently large $\alpha$,
\begin{eqnarray}\label{eq1:04}
v^*_{\alpha}(0) \le v_{\alpha}^*(1)+c \frac{f(v_{\alpha}^*(1))}{\log \alpha},
\end{eqnarray}
where $c>0$ is some constant independent of $\alpha.$
\end{lemma}
\begin{proof}
As a first step we establish an inequality similar to \eqref{eq:nedev}.

Let
\begin{eqnarray}\label{eq1:1}
\phi(y):=v_{\alpha}^*(y)-v_{\alpha}^*(1).
\end{eqnarray}
By \eqref{eq:6}, this function verifies
\begin{eqnarray}\label{eq1:2}
\left\{
\begin{array}{ll}
-\left(y\mu_{\alpha}(y)  \phi^{\prime} \right)^{\prime}=
\frac{\lambda^*}{\alpha} y\psi_{\alpha}(y) \mu_{\alpha}(y) f^{\sharp}(\phi),
& 0<y<1,\\
\phi^{\prime}(0)=0, & \phi(1)=0,
\end{array}
\right.
\end{eqnarray}
where
\begin{eqnarray}\label{eq1:3}
\mu_{\alpha}(y)=\exp\left(\int_0^y s\varphi_{\alpha}(s)ds\right).
\end{eqnarray}
Multiplying the first equation in \eqref{eq1:2} by $g^{\sharp}(\phi),$ integrating by 
parts and taking into account the boundary conditions in \eqref{eq1:2}, we get
\begin{eqnarray}\label{eq1:4}
\int_0^1 \left(f^{\sharp}_v(\phi)\ \phi^{\prime} \right)^2 d\tilde \nu=
\frac{\lambda^*}{\alpha} \int_0^1 f^{\sharp}(\phi) g^{\sharp}(\phi)d \nu,
\end{eqnarray}
where
\begin{eqnarray}\label{eq1:5}
d \tilde \nu(y) =y \mu_{\alpha}(y) dy, \qquad d \nu(y)=\psi_{\alpha}(y)d\tilde \nu(y).
\end{eqnarray}
We also note that the semi-stability condition \eqref{eq:semistab} implies that
\begin{eqnarray}\label{eq1:6}
\int_0^{\sqrt{\alpha}} (\eta^{\prime})^2 d\tilde \nu\ge \frac{\lambda^*}{\alpha} \int_0^{\sqrt{\alpha}}
f_v^{\sharp}(\phi) \eta^2 d\nu, \quad \forall \eta\in H_0^1((0,\sqrt{\alpha})).
\end{eqnarray}
Taking (in the spirit of the arguments in \cite{nedev} and \cite[Section 4.3]{stable})
\begin{eqnarray}\label{eq1:7}
\eta(y)=\left\{
\begin{array}{ll}
\tilde f^{\sharp}(\phi(y)) & 0\le y\le 1,\\
0& 1<y\le \sqrt{\alpha},
\end{array}
\right.
\end{eqnarray}
and substituting this test function into \eqref{eq1:6},
we obtain
\begin{eqnarray}\label{eq1:8}
\int_0^1 \left(f^{\sharp}_v(\phi)\phi^{\prime} \right)^2 d\tilde \nu\ge
\frac{\lambda^*}{\alpha} \int_0^1 f_v^{\sharp}(\phi) (\tilde f^{\sharp}(\phi))^2d \nu.
\end{eqnarray}
Combining \eqref{eq1:4} and \eqref{eq1:8},
we obtain
\begin{eqnarray}\label{eq1:9}
 \int_0^1 f_v^{\sharp}(\phi) (\tilde f^{\sharp}(\phi))^2d \nu \le
 \int_0^1 f^{\sharp}(\phi) g^{\sharp}(\phi)d \nu.
\end{eqnarray}

Next let $\tilde c_2>\tilde c_1.$ Then,
\begin{eqnarray}\label{eq1:10}
\int_0^1 f^{\sharp}(\phi) g^{\sharp}(\phi)d\nu=\left\{ \int_{X_1} +\int_{X_2} \right\}f^{\sharp}(\phi) g^{\sharp}(\phi)d\nu=I_1+I_2,
\end{eqnarray}
where
\begin{eqnarray}\label{eq1:11}
X_1=\{ f^{\sharp}(\phi)\le \tilde c_2 f^{\sharp}(0) \}, \quad X_2=\{ f^{\sharp}(\phi)> \tilde c_2 f^{\sharp}(0) \}.
\end{eqnarray}
By the assumption of the lemma, we have
\begin{eqnarray}\label{eq1:12}
I_2\le (1-\tilde c_0) \int_{X_2} f^{\sharp}(\phi) \tilde f^{\sharp}(\phi) f_v^{\sharp}(\phi)d\nu. 
\end{eqnarray}
Moreover, 
\begin{eqnarray}\label{eq1:13}
f^{\sharp}(\phi) \le \frac{\tilde c_2}{\tilde c_2-1} \tilde f^{\sharp}(\phi) \quad \mbox{on} \quad X_2.
\end{eqnarray}
Combining these two observations, we conclude that
\begin{eqnarray}\label{eq1:14}
I_2\le \left (1-\frac{\tilde c_0}{2}\right) \int_{X_2}  f_v^{\sharp}(\phi)\left( \tilde f^{\sharp}(\phi) \right)^2d\nu
\end{eqnarray}
provided $\tilde c_2$ is sufficiently large.

Next we observe that \eqref{eq1:9}, \eqref{eq1:10} and \eqref{eq1:14}
imply that
\begin{eqnarray}\label{eq1:15}
\frac{\tilde c_0}{2} \int_{X_2}  f_v^{\sharp}(\phi)\left( \tilde f^{\sharp}(\phi) \right)^2d\nu \le I_1
\end{eqnarray}
Noting that $f_v^{\sharp}$ is non-decreasing, we have
\begin{eqnarray}\label{eq1:16}
g^{\sharp}(t)=\int_0^t f_v^{\sharp}(s) f_v^{\sharp}(s) ds\le
f_v^{\sharp}(t) \int_0^t f_v^{\sharp}(s)  ds =f_v^{\sharp}(t) \tilde f^{\sharp}(t).  
\end{eqnarray}
Consequently,
\begin{eqnarray}\label{eq1:17}
I_1 \le \int_{X_1} f_v^{\sharp}(\phi) f^{\sharp}(\phi) \tilde f^{\sharp}(\phi)  d\nu
\end{eqnarray}
Hence, from  \eqref{eq1:15} and \eqref{eq1:17} we get
\begin{eqnarray}\label{eq1:18}
\frac{\tilde c_0}{2} \int_{X_2}  f_v^{\sharp}(\phi)\left( \tilde f^{\sharp}(\phi) \right)^2d\nu \le
\int_{X_1} f_v^{\sharp}(\phi) f^{\sharp}(\phi) \tilde f^{\sharp}(\phi)  d\nu.
\end{eqnarray}
Since
\begin{eqnarray}\label{eq1:19}
\min_{X_2} f_v^{\sharp}(\phi)\ge \max_{X_1} f_v^{\sharp}(\phi),
\end{eqnarray}
we have from \eqref{eq1:18} that
\begin{eqnarray}
\frac{\tilde c_0}{2} \int_{X_2}  \left( \tilde f^{\sharp}(\phi) \right)^2d\nu \le
\int_{X_1}  f^{\sharp}(\phi) \tilde f^{\sharp}(\phi)  d\nu,
\end{eqnarray}
and thus
\begin{eqnarray}\label{eq1:19}
\int_{X_2}  \left(  f^{\sharp}(\phi) \right)^2d\nu \le c ( f^{\sharp}(0))^2.
\end{eqnarray}
On the other hand, as follows from the definition of $X_1,$ we have
\begin{eqnarray}\label{eq1:20}
\int_{X_1}  \left(  f^{\sharp}(\phi) \right)^2d\nu \le c (f^{\sharp}(0))^2.
\end{eqnarray}
Consequently,
\begin{eqnarray}\label{eq1:21}
\int_0^1  \left(  f^{\sharp}(\phi) \right)^2d\nu \le c (f^{\sharp}(0))^2.
\end{eqnarray}
This estimate and the standard elliptic $L^p$--estimates \cite[Theorem 8.16]{GT} imply that
\begin{eqnarray}\label{eq1:22}
\phi(0)\le c \frac{\lambda^*}{\alpha} f^{\sharp}(0),
\end{eqnarray}
and therefore
\begin{eqnarray}\label{eq1:23}
v_{\alpha}^*(0)\le v_{\alpha}^*(1)+c\frac{f(v_{\alpha}^*(1))}{\log \alpha}.
\end{eqnarray}
\end{proof}

We now turn to the proof of Theorem \ref{t:3}.

\begin{proof}[Proof of Theorem \ref{t:3}]
First let us show that  the assumptions  of Lemma \ref{l:rest} 
hold under the assumptions  
of Theorem \ref{t:3}. Observe that,
as follows from the proof of Lemma \ref{l:10a},  $v_{\alpha}^*(1)\to \infty$ as $\alpha\to\infty.$
In view of this fact, the assumptions of Lemma \ref{l:rest} can be rewritten in the form:
\begin{eqnarray}\label{eq1:101a}
f(t_2)>\tilde c_1 f(t_1) \quad t_2>t_1>t_0,
\end{eqnarray}
implies
\begin{eqnarray}\label{eq1:23a}
(1-\tilde c_0) f_v(t_2) (f(t_2)-f(t_1))\ge \int_{t_1}^{t_2} \left( f_v(s)\right)^2 ds.
\end{eqnarray}
Assume that the assumptions of Theorem \ref{t:3} hold.  Choose
$t_1<\tilde t<t_2,$ $\tilde c_1>c_1$ such that
\begin{eqnarray}\label{eq1:100}
f(\tilde t)=\frac{1}{c_1} f(t_2).
\end{eqnarray}
Then, by \eqref{eqt:2}
\begin{eqnarray}\label{eq1:101}
(1-c_0) f_v(t_2) \ge f_v(\tilde t).
\end{eqnarray}
Using the monotonicity of $f_v,$  \eqref{eq1:100} and \eqref{eq1:101}, we obtain
\begin{eqnarray}
&& \int_{t_1}^{t_2} \left( f_v(s)\right)^2ds=
\left\{ \int_{t_1}^{\tilde t}+\int_{\tilde t}^{t_2} \right \}\left( f_v(s)\right)^2ds
\le f_v(\tilde t) \left( f(\tilde t) -f(t_1) \right)
+ f_v(t_2) \left( f( t_2) -f(\tilde t) \right)\le \nonumber \\
&&f_v(t_2)(f(t_2)-f(t_1)) \left\{\left(
\frac{f(t_2)-f(\tilde t)}{f(t_2)-f(t_1)}\right)+(1-c_0)\left(\frac{f(\tilde t)-f(t_1)}{f(t_2)-f(t_1)}\right)
\right\} = \\
&& f_v(t_2)(f(t_2)-f(t_1)) \left\{ 1-c_0\left(\frac{f(\tilde t) -f(t_1)}{f(t_2)-f(t_1)}\right)\right\}
\le  f_v(t_2)(f(t_2)-f(t_1)) \left\{ 1-\frac{c_0}{c_1}\left(\frac{f( t_2) -c_1 f(t_1)}{f(t_2)}\right)\right\}.
\nonumber
\end{eqnarray}
Now taking $\tilde c_1> 2 c_1$ in \eqref{eq1:101a} we get
\begin{eqnarray}\label{eq1:101b}
\left(\frac{f( t_2) -c_1 f(t_1)}{f(t_2)}\right)\ge \frac12.
\end{eqnarray}
Hence, as follows from
\begin{eqnarray}
\int_{t_1}^{t_2} \left( f_v(s)\right)^2ds\le\left( 1-\frac{c_0}{2c_1}\right) f_v(t_2)(f(t_2)-f(t_1)), 
\end{eqnarray}
which gives \eqref{eq1:23a} with $\tilde c_0=c_0/2c_1.$

We next turn to proof of \eqref{eqt:3}.  First note that taking an arbitrary  smooth test function $\eta$ with support on $[1/2,1]$
in \eqref{eq1:6} and using the monotonicity  of $f_v$ and $v_{\alpha}^*$ we have that
\begin{eqnarray}\label{eq1:24}
f_{v}(v_{\alpha}^*(1)) \le c \log \alpha.
\end{eqnarray}
Next, by convexity,
\begin{eqnarray}\label{eq1:25}
f_v(t)\ge \frac{f(t)-f(0)}{t}, \quad t>0.
\end{eqnarray}
Hence,
\begin{eqnarray}\label{eq1:26}
v_{\alpha}^*(1) f_v(v_{\alpha}^*(1))+c \ge f(v_{\alpha}^*(1)).
\end{eqnarray}
In view of \eqref{eq1:24} and \eqref{eq1:26}
we have that for sufficiently large $\alpha,$
\begin{eqnarray}\label{eq1:27}
 \frac{f(v_{\alpha}^*(1))}{\log \alpha}\le c v_{\alpha}^*(1).
\end{eqnarray}
Combining this result with \eqref{eq1:04}
we have
\begin{eqnarray}\label{eq1:27}
 v_{\alpha}^*(0)\le c v_{\alpha}^*(1),
\end{eqnarray}
which implies the result.

Finally, let us prove \eqref{eqt:5}.
Observe that \eqref{eq1:24}, \eqref{eq1:25} and Lemma \ref{l:theta} imply that
\begin{eqnarray}\label{eq1:28}
v_{\alpha}^*(0) \le c \alpha^{\gamma},
\end{eqnarray} 
for arbitrarily small $\gamma>0$.
This observation and Lemma \ref{l:10} imply that
\begin{eqnarray}\label{eq1:29}
v_{\alpha}^*(y)<
\left\{
\begin{array}{ll}
c \alpha^{\gamma} & y\in[0,\alpha^{2\delta}],\\
c & y\in[\alpha^{2\delta}, \frac{\sqrt{\alpha}}{\log \alpha}],\\
c \frac{\log (\log \alpha)}{\log{\alpha} }& y \in [ \frac{\sqrt{\alpha}}{\log \alpha},\sqrt{\alpha}],
\end{array}
\right.
\end{eqnarray}
 where $\delta >0$ is arbitrarily small.
 Hence,
 \begin{eqnarray}\label{eq1:30}
 \int_0^{\sqrt{\alpha}}( v_{\alpha}^*(y))^pydy \le c\left\{
 \frac{1}{\alpha^{1-p\gamma-4\delta}}
 +\frac{1}{(\log \alpha)^2}+\left(\frac{\log(\log \alpha)}{\log \alpha}\right)^p \right\} \alpha.
 \end{eqnarray}
 We thus have that for any $1\le p<\infty$
 \begin{eqnarray}\label{eq1:31}
 \int_B (u^*_{\alpha}(x))^pdx=\frac{2\pi}{\alpha} \int_0^{\sqrt{\alpha}}( v_{\alpha}^*(y))^pydy\le
 \left\{
 \frac{1}{\alpha^{1-p\gamma-4\delta}}
 +\frac{1}{(\log \alpha)^2}+\left(\frac{\log(\log \alpha)}{\log \alpha}\right)^p \right\}. 
 \end{eqnarray}
 In view that $\delta$ and $\gamma$ can be chosen arbitrarily small we conclude that
 the expression in the braces in \eqref{eq1:31} goes to zero as $\alpha\to \infty$.
 This observation concludes the proof of the theorem.
 \end{proof}

The following lemma gives an example of rather general nonlinearity that satisfies assumptions
of Theorem \ref{t:3}.

\begin{lemma} \label{l:proizv}
Assume that 
$f\in C^2,$  $\frac{df(s)}{ds}, \frac{d^2 f(s)}{d s^2}>0,$ and
$\frac{d^2  f(s)}{d s^2}$ is strictly increasing on
$(0,\infty).$ Then, $f$ satisfies the assumptions of Theorem \ref{t:3}.
\end{lemma}

\begin{proof}
Let us show that under the assumption of this lemma
the assumptions of Theorem \ref{t:3} are satisfied with $c_1=4$ and $c_0=\frac13.$ 

Take $t_2>t_1$ such that
\begin{eqnarray}\label{eq5:1}
f(t_2)\ge 4 f(t_1), 
\end{eqnarray}
and consider two cases: case I in which $t_2\le 2 t_1,$ and case II in which $t_2\ge 2 t_1.$

We start with case I. Assume that
\begin{eqnarray}\label{eq5:2}
f_v(t_1)\ge \frac23 f_v(t_2).
\end{eqnarray} 
Then, assumption \eqref{eq5:2} implies
\begin{eqnarray}\label{eq5:3}
\int_{t_1}^{t_2}  \frac{d^2}{ds^2} f(s)=f_v(t_2)-f_v(t_1)\le \frac13 f_v(t_2).
\end{eqnarray}
Let $t_3=2t_1-t_2$ and note that under the assumption of case I we have that $0\le t_3\le t_1.$
Since $t_1-t_3=t_2-t_1$ and $\frac{d^2 f(s)}{ds^2}$ is strictly increasing,
we have from \eqref{eq5:3} that
 \begin{eqnarray}\label{eq5:4}
\int_{t_3}^{t_1}  \frac{d^2}{ds^2} f(s)=f_v(t_1)-f_v(t_3)\le \frac13 f_v(t_2).
\end{eqnarray}
Therefore,
\begin{eqnarray}\label{eq5:5}
f_v(t_3)\ge f_v(t_1) -\frac13 f_v(t_2),
\end{eqnarray}
and hence, by \eqref{eq5:2},
\begin{eqnarray}\label{eq5:6}
f_v(t_3)\ge \frac13 f_v(t_2).
\end{eqnarray}
Since $f_v$ is an increasing function we also
have
\begin{eqnarray}\label{eq5:7}
f_v(s)\ge \frac13 f_v(t_2), \quad s\in[t_3,t_1].
\end{eqnarray}
Integrating \eqref{eq5:7} from $t_3$ to $t_1$ we have
\begin{eqnarray}\label{eq5:8}
f(t_1)-f(t_3)\ge \frac13 f_v(t_2)[t_2-t_1].
\end{eqnarray}
By convexity, 
\begin{eqnarray}\label{eq5:9}
 f_v(t_2)[t_2-t_1]\ge f(t_2)-f(t_1).
\end{eqnarray}
Thus, from \eqref{eq5:8}, \eqref{eq5:9} we obtain
\begin{eqnarray}
f(t_3)\le f(t_1)-\frac13( f(t_2)-f(t_1))=\frac13 (4 f(t_1)-f(t_2)).
\end{eqnarray}
Therefore, by \eqref{eq5:1} we have
\begin{eqnarray}
f(t_3) \le 0,
\end{eqnarray}
which contradicts the strict positivity of $f$. Hence, we must have
\begin{eqnarray}\label{eq5:10}
f_v(t_1)\le \frac23 f_v(t_2),
\end{eqnarray} 
which completes the proof in case I.

We now turn to case II. In this case by monotonicity of $\frac{d^2 f(s)}{ds^2}$
we have
\begin{eqnarray}
f_v(t_2)-f_v(0) \ge \int_0^{t_2} \frac{d^2 f(s)}{ds^2}  ds \ge \int_0^{2t_1} \frac{d^2 f(s)}{ds^2} ds
\ge 2 \int_0^{t_1} \frac{d^2 f(s)}{ds^2} ds= 2( f_v(t_1)-f_v(0)).
\end{eqnarray}
Thus, 
\begin{eqnarray}
f_v(t_2)\ge 2 f_v(t_1) -f_v(0).
\end{eqnarray}
Choosing $t_0$ large enough, so that $f_v(0) \le \frac12 f_v(t_0)$ we obtain \eqref{eq5:10}
which completes the proof.

\end{proof}

Finally, we summarize properties of extremal solutions for exponential nonlinearity.

\begin{lemma} \label{l:exp}
Let $u_{\alpha}(x)$ be an extremal solution for problem \eqref{eq:1} with $f(u)=e^u$.
Then for sufficiently large $\alpha$ we have
\begin{eqnarray}\label{eq2:1}
 \lambda^*(\alpha)=\frac{2\alpha}{\log \alpha}\left(1+O\left(\frac{1}{\sqrt{\log \alpha}}\right)\right),
\end{eqnarray}
\begin{eqnarray}\label{eq2:2}
u^*_{\alpha}(0)=O({\log \log \alpha}),
\end{eqnarray}
\begin{eqnarray}\label{eq2:3}
 \int_B (u^*_{\alpha}(x))^pdx\to 0, \quad 1\le p<\infty,
\end{eqnarray}
\begin{eqnarray}\label{eq2:4}
 \int_B \psi(x) \exp(u^*_{\alpha}(x)) dx\to \int_B \psi(x)dx.
\end{eqnarray}
\end{lemma}

\begin{proof}
We first prove \eqref{eq2:1}. We first observe that
an estimate \eqref{eq:lb} imply that for $f(u)=\exp(u)$
\begin{eqnarray}
\lambda^*(\alpha)\ge \frac{2\alpha}{\log(\alpha)}\left(1-\frac{c}{\log(\alpha)}-c\tilde R(w) \right),
\qquad \tilde R(w)=\exp(-w)+\frac{\exp(w)}{\log(\alpha)}
\end{eqnarray}
with an arbitrary  $w$. It is easy to verify that $\tilde R(w)$ attains its minimum at
\begin{eqnarray}
w=\frac{1}{2}\log (\log(\alpha)),
\end{eqnarray}
 hence
 \begin{eqnarray}
 \min_{w\in (0,\infty)}\tilde R(w)=\frac{2}{\sqrt{\log(\alpha)}}.
\end{eqnarray}
Since \eqref{eq:lb} holds for an arbitrary $w$  we have
\begin{eqnarray}\label{eq2:20}
\lambda^*(\alpha)\ge \frac{2\alpha}{\log(\alpha)}\left( 1- c\left(\frac{1}{\sqrt{\log(\alpha)}}\right)\right).
\end{eqnarray}
On the other hand in the case $f(u)=\exp(u)$ an estimate \eqref{eq:a14a} 
takes a form
\begin{eqnarray}
\lambda^*(\alpha)\le \frac{2\alpha}{\log(\alpha)}\left(1+\frac{c}{\log(\alpha)}\right).
\end{eqnarray}
Combining two estimates above we have \eqref{eq2:1}.

Next let us show that \eqref{eq2:2} holds.
By \eqref{eq1:24}
\begin{eqnarray}
\exp(v_{\alpha}^*(1))\le c\log \alpha,
\end{eqnarray}
and hence
\begin{eqnarray}
v_{\alpha}^*(1) < c \log(\log \alpha).
\end{eqnarray}
Using this observation \eqref{eq1:27} and the fact that $u_{\alpha}^*(0)=v_{\alpha}^*(0),$
we have that
\begin{eqnarray}\label{eq2:21}
u_{\alpha}^*(0) \le c \log(\log \alpha).
\end{eqnarray}
On the other hand we have from \eqref{eq2:20} and \eqref{eq2:28}
\begin{eqnarray}
\exp( -u_{\alpha}^*(0))\le c\frac{1}{\sqrt{\log \alpha}}.
\end{eqnarray}
Hence,
\begin{eqnarray}\label{eq2:22}
u_{\alpha}^*(0)\ge \frac12 \log(\log \alpha)-c.
\end{eqnarray}
Combining, \eqref{eq2:21} and \eqref{eq2:22} we have \eqref{eq2:2}.

Finally \eqref{eq2:3} and \eqref{eq2:4} follow from Theorem \ref{t:3} and Theorem \ref{t:2} respectively.
\end{proof}

\bigskip

\noindent {\bf Acknowledgements.} The work of PVG was supported, in a part,
 by Simons Foundation (Grant 317 882) and by US-Israel Binational Science Foundation (Grant 2012-057). The work of FN was partially supported by NSF (Grant DMS-1600239).
  Part of this research was carried out while VM was visiting Kent State University. 
  VM thanks the Department of Mathematical Sciences for its support and hospitality.

\end{document}